\documentclass{revtex4}
\usepackage{paralist}
\usepackage{graphics}
\usepackage{epsfig}\usepackage{amsfonts} \usepackage{amsmath}
\usepackage{amssymb} \usepackage{amsthm} 
\usepackage{mathrsfs}
\usepackage{eucal} 
\usepackage{verbatim}
\usepackage[nativepdf,colorlinks=true]{hyperref}

\usepackage[english]{babel}
\newcommand{\bauth}[1]{#1} 

\newcommand{\btitle}[1]{\emph{#1}} 
\newcommand{\bjou}[1]{#1}
\newcommand{\bvol}[1]{\textbf{#1}}
 
\newcommand{\bpage}[1]{#1} \newcommand{\byear}[1]{(#1)}
\newcommand{\bentry}[6]{\bauth{#1}, \btitle{#2}, \bjou{#3}, \bvol{#4} \byear{#6}, \bpage{#5}}
\newcommand{\bbook}[3]{\bauth{#1}, \btitle{#2} \byear{#3}}

\newcommand{\defeq}{\doteqdot} \newcommand{\naturals}{\mathbb{N}}
\newcommand{\integers}{\mathbb{Z}}
\newcommand{\reals}{\mathbb{R}} 
 \newcommand{\sone}{\mathbb{S}^1}
\newcommand{\len}{\vartheta} 

\newcommand{\um}{\underline{m}} \newcommand{\om}{\overline{m}}
\newcommand{\uM}{\underline{M}} \newcommand{\oM}{\overline{M}}
\newcommand{\uk}{\underline{k}} \newcommand{\ok}{\overline{k}}
\newcommand{\uK}{\underline{K}} \newcommand{\oK}{\overline{K}}

\newcommand{\A}{\mathscr{A}}
\newcommand{\An}[1]{\A^{#1}}
\newcommand{\Ank}[2]{\A^{#1}_{#2}}
\newcommand{\hot}{\textrm{h.o.t.}}

\newcommand{\leb}{\mathscr{L}}
\newcommand{\pr}{\mathbb{P}}

\newcommand{\de}[2]{\frac{\textup{d} #1}{\textup{d} #2}}
\newcommand{\deh}{\textup{d}}
\renewcommand{\dim}{\textrm{dim}}
\newcommand{\st}{\textrm{ s.t. }}
\newcommand{\bigo}[1]{\mathcal{O}\left(#1\right)}
\newcommand{\id}{\textrm{Id}}

\newcommand{\gs}{\tilde{\mathscr{G}}}

\newcommand{\escaping}{\mathscr{E}}

\newcommand{\fa}{\forall\,} \newcommand{\ex}{\exists\,}
\newcommand{\card}[1]{|#1|} \newcommand{\eps}{\varepsilon}

\newcommand{\ddph}{\ddot\phi}
\newcommand{\dddph}{\dddot\phi}

\newtheorem{theorem}{Theorem}[section]
\newtheorem{corollary}[theorem]{Corollary}

\newtheorem{lemma}[theorem]{Lemma}
\newtheorem{proposition}[theorem]{Proposition}

\theoremstyle{definition}
\newtheorem{definition}[theorem]{Definition}
\newtheorem{remark}{Remark}

\newtheorem*{theorema}{Theorem A}
\newtheorem*{theoremb}{Theorem B} 

\begin{document}
\title[Results in a model of Fermi acceleration]
{Stability and instability results in a model of Fermi acceleration} 
\author{Jacopo \surname{De Simoi}}
\email[]{jacopods@math.umd.edu}
\affiliation{University of Maryland, College Park, MD 20740, USA}
\begin{abstract}
We consider the static wall approximation to the dynamics of a particle bouncing on a periodically oscillating infinitely heavy plate while subject to a potential force. We assume the case of a potential given by a power of the particle's height and sinusoidal motions of the plate. We find that for powers smaller than 1 the set of escaping orbits has full Hausdorff dimension for all motions and obtain existence of elliptic island of period 2 for arbitrarily high energies for a full-measure set of motions. Moreover we obtain conditions on the potential to ensure that the total (Lebesgue) measure of elliptic islands of period 2 is either finite or infinite. 
\end{abstract}

\maketitle
\section{Introduction}
This work provides some results about the description of trajectories of a particle in a model that presents Fermi acceleration. This acceleration mechanism was originally proposed in 1949 by Enrico Fermi \cite{Enrico} to provide an explanation for presence in the universe of high energy particles known as cosmic rays. Such particles are believed to gain energy by consecutive interactions with irregularities in an otherwise stationary magnetic field. Na\"ively one would expect that a process of thermalization would occur, leading to a stationary motion of the particle itself. However, this reasoning turns out to be too simplistic and a more refined analysis shows that there is a definite probability of an average gain in energy.\\ 
In 1961, Ulam \cite{Ulam} suggested a simple Hamiltonian system to model such statistical acceleration behaviour. The model has been thereafter known as the \emph{Fermi-Ulam ping-pong model} and consists of a particle moving between two infinite walls that are performing an oscillatory motion; the particle changes its velocity only by elastic collisions with the moving walls and it is not subject to any other force. The main questions about this problem regarded the existence of trajectories with unbounded energy: such orbits can either be \emph{escaping}, i.e. such that the energy of the particle goes to infinity with time, or \emph{oscillating}, meaning that the $\limsup$ of energy is infinite while the $\liminf$ remains bounded. In 1977, KAM theory has been used to provide (\cite{Pust1}, \cite{Pust2}) the answers to such questions: for sufficiently smooth motions of the wall, all orbits are bounded, as there are invariant tori for high energies that prevent diffusion. It is interesting to note that the smoothness condition is not a mere technical issue as for less regular motions one can find (many) unbounded orbits; see for instance \cite{Zharnitsky}. \\
A variation on the same theme involves a single oscillating wall and a potential $U(x)=x^\alpha$ to bring the particle back to the wall. This model allows for more flexibility as we can vary the potential that acts on the particle.
The case of gravity potential ($\alpha=1$) has been investigated \cite{Pust1} and the study yielded the following result:
\begin{theorem}[Pustylnikov]
There is an open set of wall motions $\phi(t)$ (in the space of periodic analytic functions admitting an analytic continuation to a given strip $| \Im z |<\eps$) such that the measure of the \emph{escaping} orbits is infinite.
\end{theorem}
The case of elastic potential ($\alpha=2$) has been studied as well in \cite{Ortega1}, \cite{Ortega2}; abundance of unbounded escaping orbits has been proved under some resonance condition between the motion of the wall and the potential.\\
In a more general setting we can again use KAM theory \cite{Dima} to prove the following result:
\begin{theorem}[Dolgopyat]
If $\alpha>1$ but $\alpha\not =2$ and the motion of the wall is smooth enough, then the set of escaping and oscillatory motions is empty as there is persistence of invariant tori for high energies.
\end{theorem}
On the other hand, orbits with unbounded energy are allowed for weaker potentials. It is conjectured that for potentials weaker than gravity (i.e. for $\alpha<1$) the measure of escaping motions is zero. A first step towards this result is the following
\begin{theorem}[Dolgopyat]
If $\alpha<1/3$ and the motion of the wall is a sinusoid, then the set of escaping orbits has zero Lebesgue measure.
\end{theorem}
The above results leave several open questions regarding the largeness of the following sets:
\begin{itemize}
\item The escaping set $\escaping$ i.e. the set of orbits such that energy tends to infinity as time grows;
\item The set of orbits with bounded energy;
\item The oscillatory set i.e. the set of orbits such that $\limsup E(t)=+\infty$ and $\liminf E(t)<\infty$. 
\end{itemize}
In this paper we consider a static wall approximation to the bouncing ball system. This approximation, described in more detail in section \ref{approx}, is widely used in physics literature. It has the advantage of being given by more simple and explicit formulae whereas keeping all essential geometrical features of the complete model. 
The first result of this paper is the following 
\begin{theorema}
If $\alpha<1$ and the motion of the plate is a sinusoid, then the escaping set $\escaping$ has full Hausdorff dimension.
\end{theorema}
As a matter of fact we still ignore what should be the measure of the escaping set $\escaping$ in the case $1/3<\alpha<1$. The proof for the $\alpha<1/3$ case proceeds by defining a so-called \emph{critical set} $C$ of finite mass for some invariant measure, and such that almost every orbit passes infinitely many times through $C$; applying Poincar\'e recurrence theorem to the first return map to $C$ gives the desired result. The key ingredient of the proof is that orbits outside $C$ are strongly hyperbolic. The set $C$ has infinite measure for $\alpha>1/3$, therefore we aim to better understand dynamics on $C$ in order to obtain a smaller set $\tilde{C}$ of finite measure with the same recurrence property. As a first step we study elliptic islands inside $C$; more exactly we establish some results about the abundance of elliptic islands for high energies:
\begin{theoremb}
Let $\alpha<1$. Then in the static wall approximation, for \emph{almost all} sinusoidal motions there are elliptic islands of period 2 for arbitrarily high energies. Moreover if $2/3<\alpha<1$ we can prove the same result for \emph{all} sinusoidal motions.
\end{theoremb}
As a byproduct of the construction involved in the proof of Theorem B we obtain conditions on $\alpha$ for which the total measure of elliptic islands of fixed period 2 is either infinite or finite.
We remark that Theorem A extends without difficulties to the non-approximated bouncing balls system as we will briefly explain at the end of section \ref{hausdorffDimension}. We also expect Theorem B to remain valid for the exact bouncing ball system, but the computations needed to verify this statement are quite cumbersome. Questions regarding existence and abundance of oscillatory orbits are still completely open.
\section{The model and main approximation}
\label{approx}
We consider the problem of a point mass bouncing vertically on an infinitely heavy horizontal plate which oscillates with period 1 in the vertical direction and interacts with the particle by the law of elastic reflection. The particle is moving in a potential $U(x)=x^\alpha$, where $x$ is the vertical position and $\alpha$ is some positive real number.\\
Let $\phi(t)$ be the vertical position of the plate at time $t$, periodic of period 1; for simplicity we will consider the case:
\begin{equation}
\phi(t)=B+\frac{A}{2\pi}\sin(2\pi t).\qquad B>\frac{A}{2\pi}
\label{defphi} 
\end{equation}
It is natural to associate to the system a discrete time map defined as follows. Let $t_n$ be the time of the $n+1$st collision between the plate and the particle and $v_n$ its velocity (pointing upwards) immediately after the collision. 
As the position of the plate is a 1-periodic function of time $t$ we can consider $t_n$ on $\sone=\reals/\integers$. In this way the phase space is a half cylinder, as the velocity immediately after a collision  at time $t$ has to satisfy the inequality $v\geq\dot\phi(t)$.\\ 
An approximation common to this kind of problems is the so-called \emph{static wall approximation}, in which one considers the plate fixed at position $x_0$ but exchanging momentum with the particle as if it were moving. Notice that since the displacement of the plate is a bounded quantity, it affects the map by terms of order at most $1/v$, i.e. the accuracy of the approximation increases for high energies.\\
Define $T(v):\reals^+\rightarrow\reals^+$ as the time taken by a ball leaving $x=x_0$ with upward velocity $v$ to return on $x=x_0$ subject to the potential $U(x)$. In our case $T(v)\sim v^\gamma$, where it is easy to check that $\gamma=2/\alpha-1$. In fact, let $E_0$ be the energy of the particle after the collision:
\[
T=2\int_{x_0}^{x_{\textrm{max}}}\frac{1}{v(x)}\deh x=\int_{x_0}^{x_{\textrm{max}}}\frac{1}{\sqrt{E_0-x^\alpha}}\deh x\quad x_{\textrm{max}}=E_0^{1/\alpha}.
\]
Performing the change of variable $x=x_0^{1/\alpha}y$ we obtain:
\[
T=\int_{x_0 E_0^{-1/\alpha}}^1\frac{1}{E_0^{1/2}\sqrt{1-y^\alpha}}E_0^{1/\alpha}\deh y= E_0^{\frac{1}{\alpha}-\frac{1}{2}} \textrm{Const}\cdot\left(1+\mathcal{O}\left(\frac{x_0^\alpha}{E_0}\right)\right)\sim v^{\frac{2}{\alpha}-1}.
\]
Notice that the asymptotic expression is exact for $x_0=0$ or $\alpha=1$.\\
We will study the system in the static wall approximation and considering $T=C\cdot v^\gamma$. In this setting the map $F:(t_n,v_n)\mapsto(t_{n+1},v_{n+1})$ can be written as:
\begin{equation}
F:\left\{\begin{array}{l}t\\v\end{array}\right.\mapsto\begin{array}{l}t+T(v)\\v+2\dot\phi\left(t+T(v)\right)\end{array}\quad T(v)=C\cdot v^\gamma
\label{mainapprox}
\end{equation}
Note that if we consider $T(v)$ rather than $v$ as the fundamental variable we obtain:
\[T\mapsto C\left(\left(\frac{T}{C}\right)^{1/\gamma}+2\dot\phi\left(t+T\right)\right)^\gamma\sim T+T^{1-1/\gamma}\,C'\,\dot\phi\left(t+T\right)+\hot
\]
This map resembles the standard map, with the perturbation parameter $k=T^{1-1/\gamma}C'$ depending on $v$. We can distinguish between the following cases:
\begin{itemize}
\item \emph{weak potentials}  $0<\alpha<1 \leftrightarrow \gamma>1$: the perturbation parameter grows as velocity grows; we expect diffusion to high energies;
\item \emph{gravity} $\alpha=1 \leftrightarrow \gamma=1$: the system is equivalent to the standard map (unfolded on a semi-cylinder along $v$);
\item \emph{strong potentials} $1<\alpha<2 \leftrightarrow 0<\gamma<1$: the perturbation parameter decreases as we increase energy; there is persistence of invariant tori for high energies and therefore we can not have diffusion to arbitrarily high energies;
\item \emph{elastic potential} $\alpha=2 \leftrightarrow \gamma=0$: the period $T$ is independent of $v$, this is an exceptional case;
\item \emph{strong potentials (II)} $\alpha>2 \leftrightarrow -1<\gamma<0$: high energies correspond to small $T$, therefore we recover once more invariant tori bounding energies from above;
\item \emph{Fermi-Ulam ping pong} $\alpha\to\infty \leftrightarrow \gamma=-1$: limit case of strong potentials; once more high energies are bounded by invariant tori. Notice that in this case the static wall approximation is not anymore a good approximation as we would neglect term of the same order as $T$.
\end{itemize}
We will prove our results for the map (\ref{mainapprox}) in the weak potentials regime.
%\begin{remark}
%\end{remark} 
\section{Hausdorff Dimension of the escaping set} 
\label{hausdorffDimension}
\subsection{Main definitions}
We recall the definition of Hausdorff dimension of a metric space. First we need to define the \emph{Hausdorff $s$-measure} of a metric space $E$ as:
\[
H^s(E)\defeq\limsup_{\delta\to 0}\inf_{\substack{\mathcal{A}=\{A_i\}\\\delta-\textrm{covering of }E}}\left\{\sum_i \textrm{diam}\left(A_i\right)^s\right\}.
\]
Then we define the \emph{Hausdorff dimension} of $E$ as that critical $s$ such that:
\[
\dim_H E\defeq\inf\{s \st H^s(E)=0\}=\sup\{s \st H^s(E)>0\}.
\] 
It can be actually proved that if $s<\dim_H(E)$ then $H^s(E)=\infty$; moreover, Hausdorff dimension is a biLipschitz invariant of metric spaces.\\ 
Now let $\escaping$ be the set of escaping points, i.e. 
\[
\escaping\defeq\{(v_0,t_0)\st v_n\to\infty\ \textrm{as }n\to\infty\}.
\]
\begin{theorema}
If $\alpha<1$ and $\phi(t)$ is given by equation \eqref{defphi}, then $\dim_H \escaping =2$.
\end{theorema}
\begin{proof}
The actual proof will be provided in the two following subsections; hereby we outline all necessary steps. As in \cite{Dima} we define a set $C$ such that orbits avoiding this set are hyperbolic; we will then consider a curve $\Gamma$ obtained as a graph of a function of $t$, transverse to the stable direction on hyperbolic points. The core of the proof lies in proposition \ref{hausdim} and lemma \ref{reduction} which prove that $\dim_H\left(\escaping\cap\Gamma\right)=1$; using transversality to stable directions we obtain the desired result $\dim_H \escaping=2$.  
As trivially $\dim_H(\escaping\cap\Gamma)\leq\dim_H\Gamma=1$, it suffices to obtain the reverse inequality; such lower bound will be found by bounding the dimension of a subset containing all trajectories such that the particle gain at least energy $\epsilon$ on all (but a finite number) of collisions. The latter subset can in turn be expressed as the limit inferior of all preimages of the set of trajectories for which we gain energy at least $\epsilon$ on the first collision.\\
Let $\pi$ be the projection from the phase space onto the time coordinate $t$ and let $\gamma$ be the parametrization of $\Gamma$ that is a section for $\pi$, i.e.:
\[\pi:(t,v)\mapsto t \quad \pi\circ\gamma(t)=t.\]
We are going to consider the sequence of maps of $\sone$ to itself defined by the following composition:
\[ F_n\defeq \pi\circ F^n\circ\gamma.\]
The proof of the result for this sequence of maps relies on the fact that on the hyperbolic set the map $F$ is expanding along $\Gamma$ (and its images) with an expanding rate that increases with energy; since we are looking for points belonging to the escaping set, velocity will grow along with $n$.
Therefore we need to bound the dimension of the intersection of all preimages of a fixed set by a sequence of expanding map which expansion rates are diverging. In such a setting we mimic the calculation of the Hausdorff dimension of the usual middle third Cantor set. Roughly speaking one can define a $(n,k)$-Cantor set by dividing the unit interval in a number of equal pieces $n$, removing $n-k$ of them and then iterating the procedure on the $k$ leftover intervals. The dimension of the set obtained in this way is well-known to be:
\[
\delta_{nk}=\frac{\log{k}}{\log{n}}.
\]   
Notice that if $n$ and $k$ grow while their ratio remains bounded, the above formula for $\delta_{nk}$ approaches 1; the same conclusion holds true if we let $n$ and $k$ grow to infinity as we iterate the building procedure.
Another way to define such Cantor sets is to obtain them as intersection of preimages of a given set under an expanding map of the interval; for example the classic middle third Cantor set might be obtained by taking the intersection of all preimages of the set $[0,1/3]\cup[2/3,1]$ under the map $x\mapsto 3x\ (\textrm{mod }1)$.
If we let the expansion rate grow along with $n$ we obtain once more a set of Hausdorff dimension 1. To conclude the proof we bound in a \emph{biLipshitz} fashion the sequence of maps $\{F_n\}$ between two sequences of linear expanding maps whose expansion rates are increasing, and then obtain the corresponding estimate for the Hausdorff dimension of the set of escaping orbits.\\
In the next subsection we will prove our result for a model system; in the subsequent subsection we will show that the original system can be reduced to such model system.
\subsection{Model system}
Our model system will be a sequence of continuous functions $f_n:\sone\circlearrowleft$, such that any lift $\hat{f}_n$ satisfies the following inequalities for all real $x,y$:
\[\um_n |x-y| \leq |\hat f_n(x)-\hat f_n(y)| \leq \om_n |x-y|,\quad
1<\um_n<\om_n,\quad \um_n,\om_n \nearrow \infty.\] 
Now define
$F_n\defeq f_n\circ\cdots\circ f_1$. Then clearly we have:
\[\uM_n |x-y| \leq |\hat F_n(x)-\hat F_n(y)| \leq \oM_n |x-y|,\quad
\uM_n\defeq\prod_{k=1}^n\um_k,\quad \oM_n\defeq \prod_{k=1}^n\om_k. \]
For each $n$ we can decompose $\sone$ as a disjoint union of intervals of the form $[a,b)$ such that $F_n$ is 1-1 and onto $\sone$ on each of such intervals. Each interval corresponds to a different choice of an inverse function for $F_n$. In order to choose an inverse function for $F_n$ we have to choose an inverse branch for each one of the $f_k$ of which $F_n$ is the composition. In such a way we obtain a natural labeling for the inverse branches of $F$ where we write:  
\[
\fa n\quad \sone=\bigsqcup_{j_1\cdots j_n}I_{j_1\cdots j_n},
\]
and each $j_k$ runs on all different inverse branches of $f_k$.
\begin{figure}[!ht]
\includegraphics[height=4cm]{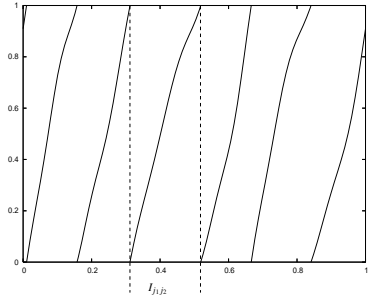}
\caption{Choosing branches of $F_k$ for $k=2$.}
\label{branches}
\end{figure}
Fix $\len\in(0,1]$ and an interval $J_0\subset\sone$ of
length $\len$, then define $J_k\defeq\left(F_k\right)^{-1}J_0=\bigsqcup J_{j_1\cdots j_k}$ where $J_{j_1\cdots j_k}= J_k\cap I_{j_1\cdots j_k}$; finally we can define:
\[J\defeq\liminf_{k\to\infty} J_k.\] 
\begin{proposition}
If $\fa n\in\naturals$ $\om_n<C\um_n$, then the Hausdorff dimension of $J$ is 1.
\label{hausdim}
\end{proposition}
\begin{proof}
 In order to compute $\dim_H J$ we will consider a smaller subset $J'$.
We will define $J'$ inductively as the limit of a decreasing sequence of sets $J'_n$. First, as $\om_n$ and $\um_n$ are increasing and diverging, from some $k$ on we have $\om_k\len>\um_k\len>2$. Without loss of generality we may assume $k=1$. Define $J'_0=J_0$; then, assuming we already defined $J'_{n-1}$ and recalling that $J_{n}=\bigsqcup J_{j_1\cdots j_{n}}$we define the decreasing sequence: 
\[
J'_n={\bigsqcup} ' J_{j_1\cdots j_n},
\]
where the union is taken \emph{only} over intervals which \emph{fully intersect} $J'_{n-1}$, i.e. all intervals that are not truncated by the intersection. In other words: each set $J_n$ is the union of many intervals; some of these intervals are completely contained in $J'_{n-1}$, and such intervals will be called \emph{good} intervals and belong to what we will call \emph{good} branches. All other intervals, i.e. such that they are not fully contained in $J'_{n-1}$ belong to \emph{bad} branches.
Now we define: 
\[
J'\defeq\liminf_{k\to\infty}J'_k=\bigcap_{k\in\naturals}J'_k.
\] 
Clearly $\fa k\ J'_k\subset J_k$. We only need to show that these sets are non-empty. Again we proceed by induction; $J'_0$ is trivially non-empty. Notice that since $F_n^{-1}(J_0)=F_{n-1}^{-1}(f_n^{-1}(J_0))$, all good intervals of $J_n$ are mapped by $F_{n-1}$ to intervals of $f_n^{-1}$ that are completely contained in $J_0$. The set $f_{n}^{-1} J_0$ is the union of several disjoint intervals and since $\um_{n}\len>2$, there exists at least one $I_{l_n}\subset J_0$, therefore at least one of the $J_l$ intervals will fully intersect $J_0$; moreover, if we let $k_n$ be the number of such intervals; we have the following estimate for $k_n$: define $\uk_n,\ok_n\in\naturals$ according to this relation:
\begin{equation}
\uk_n+1\leq\um_n\len<\uk_n+2\qquad \ok_n+1\leq\om_n\len<\ok_n+2
\label{eq:defkn}
\end{equation}
then we get $\uk_n\leq k_n\leq\ok_n+2$. In particular $k_n$ is always positive.\\
Now to obtain $J'_n$ it suffices to take, for any such interval, the preimage relative to each good branch of $F_{n-1}^{-1}$; as $J'_{n-1}$ is non-empty by inductive hypothesis we have that $J'_n$ is non-empty as well. Moreover the number $K_n$ of intervals present in $J'_n$ is easily obtained by the previous inductive procedure:
\[K_n\defeq\prod_{l=1}^nk_l.\]
\[
\prod_{l=1}^n\uk_l\defeq\uK_n\leq K_n\leq\oK_n\defeq\prod_{l=1}^n\ok_l.
\]
With a slight abuse of notation we will write from now on $J'_n=\bigsqcup_{j_1\cdots j_n}J'_{j_1\cdots j_n}$ where now each $j_l$ is running over just the good branches. The sets $J'_n$ form a decreasing sequence of nonempty
compact sets, therefore their intersection $J'$ is nonempty.\\
 We will now show that $J'\subset J$ has Hausdorff dimension $s=1$. First of all it is obvious that $s\leq 1$ being $J\subset \sone$, therefore it suffices to show that for all $s<1$ we have $H^s(J)>0$.
\begin{definition}
  We define the \emph{running Hausdorff dimension} of $\{j_1\cdots j_n\}$ the real number $s_{j_1\cdots j_n}$ such that
  \[|J'_{j_1\cdots j_n}|^{s_{j_1\cdots j_n}}=K_n^{-1}.\]
\end{definition}
The running Hausdorff dimension depends on
$\{\um_n,\om_n\}$ and $\len$ according to this estimate:
\begin{lemma}
Suppose that the following holds for all $n\in\naturals$:
\[\om_n<C\um_n.\]
Then we obtain the following lower bound for the running Hausdorff dimension:
  \begin{equation}
s_{j_1\cdots j_n}>1+\frac{\log\len+n\left(\log\frac{\len}{3}-\log C\right)}{\log\oM_n-\log\len}\defeq 1-\eps_n\nearrow 1 \textrm{ as } n\rightarrow\infty
    \label{eq:s_estimate}
  \end{equation}
\end{lemma}
\begin{proof}
  Using (\ref{eq:defkn}) we get
  \[
  \uk_n\leq\len \um_n\leq\uk_n+2\leq 3\uk_n\quad \ok_n\leq \len\om_n\leq \ok_n+2 \leq 3\ok_n.
  \]
  \[
  \uK_n\leq\len^n \uM_n\leq 3^n\uK_n\quad \oK_n\leq\len^n\oM_n\leq 3^n \oK_n.
  \]
Now since we know that $\len/\oM_n<|J'_{j_1\cdots j_n}|$, we obtain:
\[
\uK_n\left(\len/\oM_n\right)^{s_{j_1\cdots j_n}}<K_n \left(\len/\oM_n\right)^{s_{j_1\cdots j_n}}<K_n |J'_{j_1\cdots j_n}|^{s_{j_1\cdots j_n}}=1,
\]
and taking logarithms we establish the following inequality:
\[\log \uK_n-s_{j_1\cdots j_n}\left(\log\oM_n-\log\len\right)<0.\]
Therefore we obtain the bound:
\[
s_{j_1\cdots j_n}>\frac{\log\uM_n+n\log\frac{\len}{3}}{\log\oM_n-\log\len}.
\]
Now by hypothesis we know that $\oM_n<C^n\uM_n$, therefore $\log\oM_n<n\log C+\log\uM_n$ that in turn implies:
\[
s_{j_1\cdots j_n}>\frac{\log\oM_n+n\left(\log\frac{\len}{3}-\log C\right)}{\log\oM_n-\log\len},
\]
that implies estimate (\ref{eq:s_estimate}) provided that we show that $\eps_n$ is going to 0; but
\[
\eps_n\sim\frac{n}{\log \oM_n}\rightarrow 0,
\]
as $\log \oM_n/n$ is the average of the diverging sequence $\log \om_n$.
\end{proof}
To obtain a lower bound on the Hausdorff dimension we are going to use the following proposition (see e.g. \cite{Falconer}):
\begin{proposition}
\label{prop_meas}
Suppose there exists a probability measure $\mu$ on a metric space $X$ on the $\sigma$-algebra of Borel sets such that for all sufficiently small balls $B$ we have:
\begin{equation}
\mu(B)<C \textrm{diam}(B)^s,
\label{measure_estimate}
\end{equation}
then $\dim_H X>s$
\end{proposition}
We will now prove the following 
\begin{proposition}
There exists a probability measure $\mu$ on $J'$ satisfying \eqref{measure_estimate} for all $s<1$.
\end{proposition}
\begin{proof}
First we build a probability measure $\mu$ on $J'$, and then check that $\mu$ satisfies \eqref{measure_estimate} for every $s<1$. For each $n$ and choice of $j_1\cdots j_n$, fix a point $x_{j_1\cdots j_n}\in J'_{j_1\cdots j_n}$. Then define the following sequence of positive functionals acting on $\mathcal{C}(\sone,\reals)$:
\[
\fa\varphi\in\mathcal{C}(\sone,\reals)\quad \Phi_n(\varphi)\defeq\sum_{j_1,\cdots,j_n}\frac{1}{K_n}\varphi\left(x_{j_1\cdots j_n}\right).\]
Now we argue that this sequence of functionals has a weak limit for $n\to\infty$. In fact any continuous function $\varphi$ on $\sone$ is also uniformly continuous; therefore $\fa\epsilon\ \ex\delta$ such that $|x-y|<\delta$ implies $|\varphi(x)-\varphi(y)|<\epsilon$. Now take $n$ such that $\max_{j_1\cdots j_n}|J'_{j_1\cdots j_n}|<\delta$. Then for each $m>n$:
\begin{eqnarray*}
\left|\Phi_n(\varphi)-\Phi_m(\varphi)\right|&=&\left|\sum_{j_1,\cdots,j_n}\frac{1}{K_n}\varphi\left(x_{j_1\cdots j_n}\right)-\sum_{j_1,\cdots,j_m}\frac{1}{K_m}\varphi\left(x_{j_1\cdots j_m}\right)\right|=\\
&=&\left|\sum_{j_1,\cdots,j_n}\left(\frac{1}{K_n}\varphi\left(x_{j_1\cdots j_n}\right)-\sum_{j_{n+1},\cdots,j_m}\frac{1}{K_m}\varphi\left(x_{j_1\cdots j_m}\right)\right)\right|\leq\\
&\leq&\sum_{j_1,\cdots,j_m}\frac{1}{K_m}\left|\varphi\left(x_{j_1\cdots j_n}\right)-\varphi\left(x_{j_1\cdots j_m}\right)\right|\leq\\
&\leq&\sum_{j_1,\cdots,j_m}\frac{1}{K_m}\epsilon=\epsilon.
\end{eqnarray*}
In the above inequalities we used the fact that by definition \smash{$\sum_{j_{n+1}\cdots j_m}K_n/K_m=1$} and that $J'_{j_1\cdots j_m}\subset J'_{j_1\cdots j_n}$ which implies $x_{j_1\cdots j_n},x_{j_1\cdots j_m}\in J'_{j_1\cdots j_n}$ and as $|J'_{j_1\cdots j_n}|<\delta$, the inequality follows. The sequence $\Phi_n$ weakly converges to a positive functional $\Phi$ i.e. to a Borel measure $\mu$ via the Riesz representation theorem. Moreover $\mu$ is a probability measure (it suffices to compute the limit against the function $\varphi\equiv 1$).\\
At this point for any Borel set $E$ and $n\in\naturals$ we can write:
\begin{eqnarray*}
  \mu(E)&=&\lim_{m\rightarrow\infty} \sum_{j_1,\cdots, j_m}\frac{1}{K_m}\chi_E\left(x_{j_1\cdots j_m}\right)\\
&=&\sum_{j_1,\cdots, j_n}\lim_{m\rightarrow\infty}\sum_{j_{n+1},\cdots, j_m}\frac{1}{K_m}\chi_E\left(x_{j_1\cdots j_nj_{n+1}\cdots j_m}\right)\\
&\leq&\sum_{\substack{j_1,\cdots, j_n\\J'_{j_1\cdots j_n}\cap E\not =\emptyset}}\frac{1}{K_n}\lim_{m\rightarrow\infty}\sum_{j_{n+1},\cdots, j_m}\frac{K_n}{K_m}=\sum_{\substack{j_1,\cdots, j_n\\J'_{j_1\cdots j_n}\cap E\not =\emptyset}}\frac{1}{K_n},
\end{eqnarray*}
By  definition of running Hausdorff dimension this implies the following estimate for $\mu(E)$: 
\begin{equation}
\mu(E)\leq\sum_{\substack{j_1,\cdots, j_n\\J'_{j_1\cdots j_n}\cap E\not =\emptyset}}\left|J'_{j_1\cdots j_n}\right|^{s_{j_1\cdots j_n}}
\label{covariance}
\end{equation}
In order to obtain estimate (\ref{measure_estimate}), let $\rho$ be the radius of the ball $B$. Then
$2\rho\in[\uM_n^{-1},\uM_{n-1}^{-1})$ for some $n$. Now subdivide
$[\uM_n^{-1},\uM_{n-1}^{-1})$ in $\um_n$ intervals each of length
$\uM_n^{-1}$. Let $l>0$ be such that $2\rho\in[l
\uM_n^{-1},(l+1)\uM_n^{-1})$. This means that
\[
\frac{1}{l+1}2\rho\leq \uM_n^{-1} < \frac{1}{l}2\rho.
\]
Using this estimate we know that each interval of $J'_n$ contains
a ball of radius $\rho/(l+1)$, therefore a ball of radius $\rho$ can
intersect at most $l+1+2$ such intervals. Using (\ref{covariance}) we
get:
\begin{eqnarray*}
\mu(B)&\leq&\sum_{\substack{j_1,\cdots, j_n\\J'_{j_1\cdots j_n}\cap B\not =\emptyset}}\left|J'_{j_1\cdots j_n}\right|^{s_{j_1\cdots j_n}}<\sum_{\substack{j_1,\cdots, j_n\\J'_{j_1\cdots j_n}\cap B\not =\emptyset}}\left|J'_{j_1\cdots j_n}\right|^{1-\eps_n}\\
&\leq&\sum_{\substack{j_1,\cdots, j_n\\J'_{j_1\cdots j_n}\cap B\not =\emptyset}}\uM_n^{-(1-\eps_n)}\leq (l+3)\uM_n^{-(1-\eps_n)}\\
&\leq& (l+3)
\left(\frac{2\rho}{l}\right)^{1-\eps_n}=\frac{l+3}{l^{1-\eps_n}}\card{B}^{1-\eps_n}
\end{eqnarray*}
If $n$ is big enough, the fraction equals $l^{\eps_n}+\bigo{1}$, which a priori may be unbounded; but we have:
\[
l^{\eps_n}<\um_n^{\epsilon_n} < \uM_n^{\eps_n},
\]
and:
\begin{eqnarray*}
\uM_n^{\eps_n}&=&\exp\left(\log\uM_n\frac{\log\len+n\left(\log\frac{\len}{3}-\log C\right)}{\log\oM_n-\log\len}\right)\\
&=&\exp\left(\frac{\log\uM_n-\log\len+\log\len}{\log\oM_n-\log\len}\left(\log\len+n\left(\log\frac{\len}{3}-\log C\right)\right)\right)\\
&\leq&\exp\left(\left(1+\frac{\log\len}{\log\oM_n-\log\len}\right)\left(\log\len+n\left(\log\frac{\len}{3}-\log C\right)\right)\right)\\
&=&\exp\left(\log\uM_{n-1}\left(1+\frac{\log\len}{\log\oM_n-\log\len}\right)\frac{\log\len+n\left(\log\frac{\len}{3}-\log C\right)}{\log\uM_{n-1}}\right)\\
&\defeq&\uM_{n-1}^{\eps'_n},
\end{eqnarray*}
and $\eps'_n\rightarrow 0$ as $n\rightarrow \infty$. Then 
\[
l^{\eps_n}<\uM_{n-1}^{\eps'_n}<(2\rho)^{-\eps'_n}.
\]
So we finally obtain $\mu(B)<C\card{B}^{1-\eps_n-\eps'_n}$. This estimate
still depends on $n$ therefore on $\card{B}$, but notice that
$\eps_n+\eps'_n$ is monotone decreasing to $0$, therefore if we fix $n$ the
inequality will hold for all $B$ such that $\card{B}<\uM_{n-1}^{-1}$.  At
this point it is easy to see that $\fa \epsilon>0\ \ex \bar{\delta}$ such
that any $\delta$-ball $B$ with $\delta<\bar{\delta}$ will satisfy inequality \eqref{measure_estimate}:
\[
\mu(B)<C\textrm{diam}(B)^{1-\epsilon}.
\]
\end{proof}
For all $\epsilon>0$ proposition \ref{prop_meas} gives $\dim_HJ'>1-\epsilon$, therefore \[\dim_H J\geq\dim_HJ'=1,\] that proves proposition \ref{hausdim}.
\end{proof}
Now notice the very important fact that all estimates required in the proof are evaluated at points of the set $J'$; the assumptions of the statement can therefore be relaxed and we can require them to be true just on points belonging to such set. Moreover a corollary to the proof of proposition \ref{hausdim} is the somewhat weaker statement:
\begin{corollary}
Let $J_n$ be as before. If $\om_1\len>\um_1\len>2$ and $\om_n<C\um_n$, for all $n$ then: 
\[
\tilde{J}\defeq\bigcap_{k\in\naturals}J_k\quad\dim_H \tilde{J}=1.
\]\label{weaker}
\end{corollary}  
In the following section we shall use only the weaker version we just stated. 
\subsection{Reduction to the model system}
In this subsection we show that the results found in the previous subsection can be applied to the approximated bouncing ball system. We first build a set such that the dynamics of orbits that never leave this set is hyperbolic; in doing so we closely follow \cite{Dima}.
 For small $a$ we define the set 
\[
B_a\defeq\left\{(t,v) \st |\ddot\phi(t)|>a\right\},
\]
and compute the differential $\deh F$:
\[  
\deh F=\left(
\begin{array}{cc}
1&T'(v)\\
2\ddot\phi(t+T(v)) &1+2 T'(v) \ddot\phi(t+T(v))
\end{array}
\right).
\]
\begin{lemma}
If $v$ is big enough and $F(t,v)\in B_a$ then $\deh F$ is hyperbolic
\label{hyp}
\end{lemma}
\begin{proof}
It suffices to check that $\textrm{Tr}(\deh F)=2(1+T'(v)\ddot\phi(t+T(v)))>2$, but for big enough $v$, $T'(v)\gg 1$, therefore, since by hypothesis $|\ddot\phi(t+T(v))|>a$ we have hyperbolicity.
\end{proof}
Now we want to find an invariant cone field. In order to do so we find the direction corresponding to the expanding eigenvector in the limit $T'\to\infty$, corresponding to high energies. We claim that a small cone around this direction is invariant for $v$ large enough. In fact the eigenvectors of $\deh F$ in the above limit are 
\begin{eqnarray*}
V_+=(\delta t,\delta v)&=&(1,2 \ddot\phi(t+T(v)))\\
V_-=(\delta t,\delta v)&=&(1,0)
\end{eqnarray*}
Therefore, having fixed a small $0<c<a$, the cone field defined on $B_a$ by the following expression:
\begin{equation}
\mathcal{C}_{t,v}\defeq\{(\delta t,\delta v)\st \left| \frac{\delta v}{\delta t}-2 \ddot\phi(t+T(v))\right|<c\}
\label{conefield}
\end{equation}
is invariant for large enough $v$ as $V_+$ and $V_-$ are well separated on $B_a$. This means that if we take a curve whose tangent vectors lie in such cone field and we apply $F$ we are going to obtain (on the hyperbolic set) a curve whose tangent vectors again lie in the cone field; moreover since the tangents to the curve are close to the expanding direction of the map, the dynamics will be expanding along the curve. For such reasons we now define 
\[\Gamma_C=\{(t,v)=(t,C+2\dot\phi(t))\},\]
and curves from this family with big enough $C$ will be good candidates for our purposes.\\
The condition that an orbit never leaves the set $B_a$ controls the expansion rates of the maps $F_n$. Now we need to define a set to ensure the escaping condition. As we mentioned in the outline we now define the following set:
\[A_{\eps}\defeq\left\{\left(t_n,v_n\right)\st\ v_{n}-v_{n-1}>\eps\right\}.\]
As $v_{n}-v_{n-1}=2\dot\phi(t_{n})$, we have that $\left\{2\dot\phi(t)>\eps\right\}= A_\eps$. We can therefore select values of $a$, $\eps$ and $\bar v$ such that there exists an interval $J\subset \sone$ such that 
\[
J_0\times \left\{v\geq \bar{v}\right\} \subset A_\eps \cap B_a,
\]  
and $\bar v$ is big enough for lemma \ref{hyp} to hold true and for the cone field in (\ref{conefield}) to be invariant for any $v\geq \bar v$.
Recall now the definition of $F_n$:
\[\gamma(t):\sone\ni t\mapsto (t,v=C+2\dot\phi(t))\quad
\pi:(t,v)\mapsto t. \]
\[ F_n\defeq \pi\circ F^n\circ\gamma.\] 
\begin{lemma}
\label{reduction}
For large enough $C$ there exist positive constants $\tilde{C}_1, \tilde{C}_2,\underline{C},\overline{C}$ such that:
\[
\fa t\in \bigcap_{k=0}^n F_k^{-1}(J_0),\quad \tilde{C}_1\prod_{k=0}^n\left(\underline{C}+\eps k\right)^{\gamma-1}<\left|F'_n(t)\right|<\tilde{C}_2\prod_{k=0}^n\left(\overline{C}+3Ak\right)^{\gamma-1}.
\]
\begin{proof}
  Let $F^k(t,\gamma(t))\defeq (t_k,v_k)$. First we get an estimate for $\de{t_{k+1}}{t_k}$; we know that $t_{k+1}=t_{k}+T(v_k)$, then:
\[
\de{t_{k+1}}{t_k}=1+T'(v_k)\frac{\deh v_k}{\deh t_k}.
\]
If $t\in \bigcap F_k^{-1}(J)$ and $C$ is big enough, we know that the cone field $\mathcal{C}$ is invariant i.e. $(1,\de{v_k}{t_k})\in\mathcal{C}$, therefore 
\[
\left|\de{v_k}{t_k}-2\ddot\phi(t_{k+1})\right|<c,
\]
that implies
\[
\de{t_{k+1}}{t_k}=\bigo{1}+\textrm{Const }v_k^{\gamma-1} (\ddot\phi(t_{k+1})+\bigo{c}).
\]
and since $\left(t_k,v_k\right)\in B_a$ and $c<a$ we can find positive $\tilde{C_1}$ and $\tilde{C_2}$ such that
\[
\tilde{C_1}v_k^{\gamma-1}<\left|\de{t_{k+1}}{t_k}\right|<\tilde{C_2}v_k^{\gamma-1}.
\]
Now since each $\left(t_k,v_k\right)\in A_\eps$ we have bounds on $v_k$, i.e.
\[
\underline{C}+\eps k<v_k<\overline{C}+3A k,
\]
where $\underline{C}$ and $\overline{C}$ are respectively the minimum and the maximum $v$ of the curve $\Gamma$ and by \eqref{defphi}, $A$ is the maximum of $\dot\phi(t)$.
Therefore 
\[
\tilde{C_1}\left(\underline{C}+\eps k\right)^{\gamma-1}<\left|\de{t_{k+1}}{t_k}\right|<\tilde{C_2}\left(\overline{C}+3A k\right)^{\gamma-1}.
\]
Then by the chain rule we get the required inequality.
\end{proof}
\end{lemma}
In order to apply the results found in the previous section we have to ensure that $\tilde{C_1}\underline{C}^{\gamma-1}\cdot|J_0|>2$; this can always be obtained by choosing $C$ large enough. Now we can apply corollary \ref{weaker} and obtain $\dim_H(\escaping\cap J)=1$.
\end{proof}
As a final remark we would like to point out that the same idea works for the non-approximated map with some slight modifications; the proof has not been included in this paper for sake of simplicity.
\section{Existence of elliptic islands for arbitrarily high energy}
In this section we are going to prove Theorem B:
\begin{theoremb}
Let $\alpha<1$. Then in the static wall approximation, for almost all sinusoidal motions there are elliptic islands of fixed period 2 for arbitrarily high energies. Moreover if $2/3<\alpha<1$ we can prove the same result for \emph{all} sinusoidal motions.
\end{theoremb} 
First we recall the definition of elliptic island: if an elliptic fixed point $p$ for a two-dimensional symplectic map $F$ is surrounded by a invariant set of closed curves on which the dynamics is conjugated to an irrational rotation on the circle we say that $p$ is surrounded by an elliptic island. Such islands are obviously Lyapunov stable.\\ 
The outline of the proof of Theorem B is as follows. In section \ref{construction} and \ref{period2orb} we build a reversor map by exploiting a symmetry of the system; we recall that a reversor is an idempotent map that conjugates the dynamics with its inverse. Following a standard technique in the theory of reversible maps (see e.g. \cite{LambRoberts}), we use the locus of fixed points of the reversor map to find a number of periodic orbits; most of them will be hyperbolic but by fine-tuning the  amplitude $A$ of the oscillation of the plate we can turn some of them into elliptic periodic points. In the main approximation (defined in section \ref{approx}) it is easy to state the ellipticity condition (section \ref{ellipticity}) in terms of $A$. We can actually find conditions to ensure that the multiplier of such periodic points belong to some given sub-interval of $\sone$. This fact will turn out to be useful when dealing with resonances. Such conditions, along with a non-degeneracy condition on the Birkhoff normal form that we prove in section \ref{nondegeneracy}, is sufficient to establish the presence of an elliptic island around the periodic point (see for instance \cite{laz} or \cite{Rafa}). The ellipticity condition (section \ref{ellcondition}) turns out to be an arithmetic condition on the parameter $A$ that is satisfied (section \ref{arith}) by infinitely many periodic points for a set of full measure of $A$ for all weak potentials. The same proof gives the stronger result that for $2/3<\alpha<1$ the statement is true for \emph{all} parameters $A$.\\ The techniques developed in the following sections allow us obtain some interesting quantitative results. As an example, in proposition \ref{islandmeasure}, we state conditions on $\gamma$ guaranteeing that the Lebesgue measure of the elliptic islands obtained with our construction is either infinite or finite.
\subsection{Construction of periodic orbits} 
\label{construction} 
Recall first the definition of $F$: 
\begin{equation*}
F:\left\{\begin{array}{l}t\\v\end{array}\right.\mapsto\begin{array}{l}t+T(v)\\v+2\dot\phi\left(t+T(v)\right)\end{array}\quad T(v)=C\cdot v^\gamma.
\end{equation*}
We are going to find periodic orbits by building a set with special dynamical properties and then considering intersections with its forward and backward images.\\
\begin{remark}
If $\dot\phi$ is odd with respect to some point $t_0$, then $F$ has a \emph{reversor} map $R$ such that 
\[
R^2=\id\qquad RFR=F^{-1}.
\]
We can explicitly write $R$ as follows:
\[
R:(t,v)\mapsto\left(2t_0-t-T(v),v\right).
\]
If we square the map we get the identity and it is an easy check that $R$ conjugates $F$ with its inverse.
Notice also that being defined on a cylinder, if $\dot\phi$ is odd with respect to $t_0$ it has to be odd with respect to $t_0+\frac{1}{2}$ as well. 
\end{remark}
We are going to define the set $\ell$ (for \emph{locus}) of fixed points of $R$:
\[ \ell\defeq\{(t,v) \st\ R(t,v)=(t,v)\}.\]
The set $\ell$ is the disjoint union of the following two curves:
\begin{equation}
\begin{array}{rcl}
\ell_+(v)&=&\left(t_0-\frac{1}{2}T(v),v\right)\\
\ell_-(v)&=&\left(\frac{1}{2}+t_0-\frac{1}{2}T(v),v\right).
\end{array}
\end{equation}
Such curves wind around the cylinder as $v$ increases. It is more convenient to partition such curves in pieces that wind just once around the cylinder in order to get graphs of (single valued) functions of $t$. This can be easily done by inverting the 1-1 map $v\mapsto T$; let this inverse be $v(T)$. Define now:
\[
\fa n\in\naturals\quad\ell_n(t)=\left(t,v\left(2\left(t_0-t\right)+n\right)\right).
\]
The curve $\ell_+$ corresponds to even values of $n$ while $\ell_-$ to odd values. Subscripts will always refer to branches and superscripts will always refer to iterates of the set, i.e. for $k\in\integers$, $\ell^k_n\defeq F^k\ell_n$.\\
The important dynamical property of $\ell$ is that 
\[
\fa x\in\ell\quad F^kx=RF^{-k}Rx=RF^{-k}x,
\]
therefore, if $F^kx$ belongs to $\ell$ as well, we get $F^kx=F^{-k}x$ that implies that the orbit of $x$ is periodic  of (possibly not least) period $2k$. Therefore points belonging to $\ell^k\cap\ell$ for $k\not=0$ are periodic points.   
The issue is now to figure out when the corresponding periodic orbits are elliptic or hyperbolic. Taking inspiration from \cite{gl} we work out from scratch the period 2 case.
\subsection{Period 2 orbits}
\label{period2orb}
First we classify period 2 orbits. This turns out to be quite simple, as the following proposition shows. To fix notations, let $\{p_1,p_2\}$ be a 2-periodic orbit, $p_1=(t_1,v_1)$ and $p_2=(t_2,v_2)$; since $v_0=v_2$, we have $\dot\phi(t_1)=-\dot\phi(t_2)$.  
\begin{definition} Being $\phi$ is a sinusoid, 2-periodic orbits can only be of one of the following types:
\begin{itemize}
\item $\ddot\phi(t_1)=\ddot\phi(t_2)$, such an orbit will be called a \emph{$(+)$-orbit};
\item $\ddot\phi(t_1)=-\ddot\phi(t_2)$, such an orbit will be called a \emph{$(-)$-orbit};
\end{itemize}
\end{definition}
\begin{proposition} 
Let $\{p_1,p_2\}$ be a 2-periodic orbit; there can be two cases:
\begin{itemize}
\item $p_1,p_2 \in\ell$, the orbit is a $(+)$-orbit;
\item $T(v_1)\equiv T(v_2)\equiv \frac{1}{2} \mod 1$; the orbit is a $(-)$-orbit.
\label{period2}
\end{itemize} 
\end{proposition}
\begin{proof}
Let us write the condition for $p_1=(t_1,v_1)$ and $p_2=(t_2,v_2)$ to be periodic: 
\[
\begin{array}{rcl}
t_1&\mapsto&t_1+T(v_1)=t_2\\
t_2&\mapsto&t_1+T(v_1)+T(v_2)=t_1\\
v_1&\mapsto&v_1+2\dot\phi(t+T(v_1))=v_2\\
v_2&\mapsto&v_1+2\left(\dot\phi(t+T(v_1))+\dot\phi(t+T(v_1)+T(v_2))\right)=v_1.
\end{array}
\]
So that we have the two conditions:
\[
T(v_1)+T(v_2)\equiv 0 \mod 1\quad \dot\phi(t+T(v_1))+\dot\phi(t)=0.
\]
As $\dot\phi(t)$ is a co-sinusoid, the second condition can be true only under one of the two following assumptions:
\begin{itemize}
\item $t_1+T(v_1)=2t_0-t_1$ therefore $p_1\in\ell$; the same is true for $p_2$
\item $t_1+T(v_1)=t_1+\frac{1}{2}$ therefore $T(v_1)\equiv T(v_2) \equiv \frac{1}{2}$.
\end{itemize}
Notice moreover that as $\dot\phi$ is odd with respect to $t_0$, $\ddot\phi$ is even with respect to the same point $t_0$, therefore orbits of the first kind are $(+)$-orbits. On the other hand, orbits of the second type satisfy the opposite condition $\ddot\phi(t_1)=-\ddot\phi(t_1+1/2=t_2)$ and so they are $(-)$-orbits.
\end{proof}
\subsection{Elliptic locus for period 2 orbits}
\label{ellipticity}
In this section we find a subset of the phase space such that all 2-orbits that lie in the set are elliptic. 
\begin{proposition}
Let the two points of the orbit be $(t_1,v_1)$ and $(t_2,v_2)$; let $T'_i\defeq T'(v_i)$, but $\dot\phi_1\defeq \dot\phi(t_2)$, $\dot\phi_2\defeq \dot\phi(t_1)$ and similarly for $\ddot\phi$. Notice that by proposition \ref{period2} we have $\ddot\phi_1=\pm\ddot\phi_2$.
Let $\nu_i\defeq \ddot\phi_i T'_i$ and $-1\leq c_1<c_2\leq 0$; define the following sets:
\begin{eqnarray*}
&E^+_{c_1c_2}\defeq\{(\nu_1,\nu_2)\st \nu_1+\nu_2+\nu_1\nu_2\in(c_1,c_2)\}\\
&E^-_{c_1c_2}\defeq\{(\nu_1,\nu_2)\st \nu_1\nu_2\in(c_2,c_1)\}
\end{eqnarray*}
Then $(\pm)$-orbits belonging to $E^\pm_{c_1c_2}$ are elliptic with multiplier $\lambda$ such that \protect{$\Re\lambda\in(1+2c_1,1+2c_2)$}.
\label{therecanbeonlytwo}
\end{proposition}
\begin{proof}
We start by computing the differentials $\deh F$ on each point on the orbit:
\[
\deh F_1=\left(
\begin{array}{cc}
1&T'_1\\
2\ddot\phi_1 &1+2\ddot\phi_1 T'_1
\end{array}
\right)
\qquad
\deh F_2=\left(
\begin{array}{cc}
1&T'_2\\
2\ddot\phi_2 &1+2\ddot\phi_2 T'_2
\end{array}
\right).
\]
The condition for any matrix $M$ to be elliptic is that $|\textrm{Tr}M|<2$; moreover, for any elliptic matrix $M$, $\frac{1}{2}\textrm{Tr}M$ is the real part of its multiplier. By direct computation of the trace of the product of the two differentials and using the $(\pm)$ relations on $\ddot\phi$, we obtain:
\begin{equation}\tag{+}
\frac{1}{2}Tr\left(\deh F_1 \deh F_2\right)=1+2\left(\ddot\phi_1 T'_1+\ddot\phi_2 T'_2\right)+2T'_1T'_2\ddot\phi_1\ddot\phi_2.
\end{equation}
\begin{equation}\tag{-}
\frac{1}{2}Tr\left(\deh F_1 \deh F_2\right)=1+2T'_1T'_2\ddot\phi_1\ddot\phi_2.
\end{equation}
By direct computation we obtain the following conditions in $\nu_i$ that ensure ellipticity and the supplementary condition on the multiplier:
\begin{equation}\tag{+}
\nu_1+\nu_2+\nu_1\nu_2\in (c_1,c_2)
\end{equation}
\begin{equation}\tag{-}
\nu_1\nu_2\in (c_1,c_2),
\end{equation}
that are the defining conditions for the sets $E^\pm_{c_1c_2}$.
\end{proof}
Notice that since $\ddot\phi_1=\pm\ddot\phi_2$ we have that 
\[
\frac{\nu_1}{\nu_2}=\frac{\ddot\phi_1T'_1}{\ddot\phi_2T'_2}=\pm\frac{T'_1}{T'_2}.
\]
Since $T'=T^{1-1/\gamma}$ and $\left| T_2^{1/\gamma} - T_1^{1/\gamma}\right |\sim A=\max\left(\dot\phi\right)$ we get
\[
\frac{\nu_1}{\nu_2}\sim\pm\left(\frac{T_1}{T_2}\right)^{1-\frac{1}{\gamma}}\rightarrow\pm 1^- \textrm{ as } T_1\rightarrow\infty.
\]
Even if the shape of the sets $E^\pm_{c_1c_2}$ is not very complicated, it is convenient to state a sufficient condition in terms of just one parameter $\nu$. Fix $\eps$ small, then if we let $v_1<v_2$ big enough, we have  $(1-\eps)|\nu_2|<|\nu_1|<|\nu_2|$; a direct calculation yields the following sufficient conditions for 2-orbits to be elliptic satisfying the required condition on the multiplier:
\begin{equation}
 (+)\textrm{-orbits: }\nu_2\in (c_1',c_2') \qquad (-)\textrm{-orbits: }|\nu_2|\in(|c_1''|,|c_2''|)
\label{sufficient}
\end{equation}
where $c_1',c_2'$ and $c_1'',c_2''$ are $\eps$-close to $c_1$ and $c_2$. Note the we can get \emph{all} elliptic orbits just by taking $c_1=-1$ and $c_2=0$.
\subsection{Description of $\ell^1$ and ellipticity condition}
\label{ellcondition}
In this and the subsequent sections all pictures and geometric constructions are made keeping in mind the coordinates $(t,T(v))$. In such coordinates $\ell$ is represented by a straight line, and it is much easier to have geometric intuition about the dynamics. Recall that the simple choice for $\phi$ given by equation \eqref{defphi} implies the following expression for $\dot\phi$:
\[
\dot\phi(t)=A\cos(2\pi t).
\]
Now we claim the following
\begin{proposition}
Let $v(T)$ be the inverse function of $T(v)$. Then there exist real positive numbers $C_1<C_2$ such that, for any $n<m\in\naturals$ the following condition
\[
v\left(m-\frac{1}{2}-\frac{C_2}{A}\frac{1}{m^{1-1/\gamma}}\right)<v\left(n+\frac{1}{2}\right)+2A<v\left(m-\frac{1}{2}-\frac{C_1}{A}\frac{1}{m^{1-1/\gamma}}\right)
\]
implies the existence of a 2-periodic elliptic point close to $T=n+1/2$ and $T=m+1/2$ such that its multiplier satisfies \protect{$\Re\lambda\in(1+2c_1,1+2c_2)$}.
\label{condition}
\end{proposition}
\begin{proof}
Consider $\ell^1_n$: let $v_k(\tau)\defeq v\left(2\left(\tau-t_0\right)+k\right)$ so that:
\[ \ell^1_n(\tau)=\left(\tau,v_n(\tau)+2\dot\phi(\tau)\right).\] 
So for $A=0$ this is just a line in the $(t,T(v))$-plane.
\begin{figure}[!h]
  \begin{center}  
    \includegraphics[width=4cm]{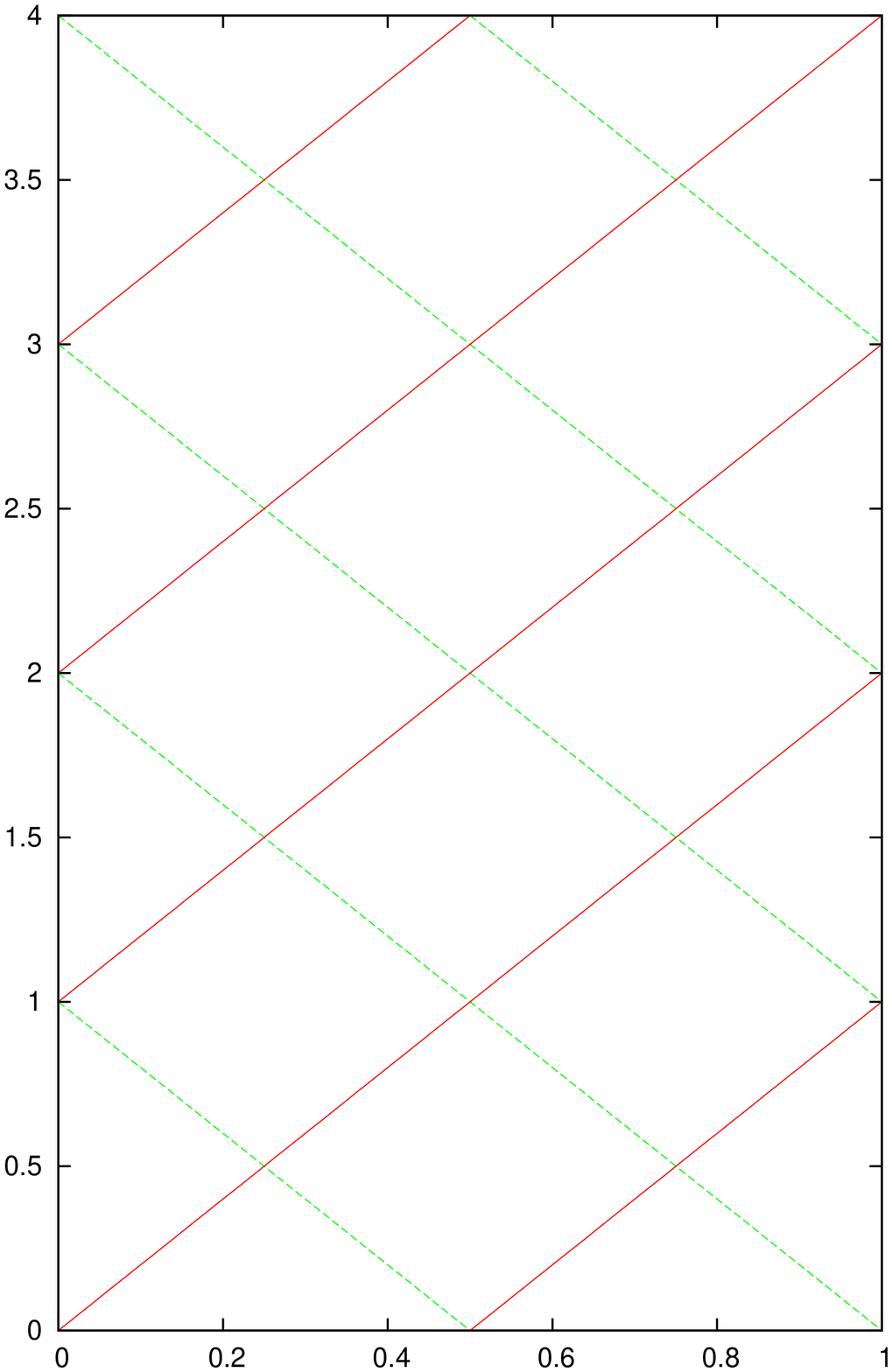} 
    \includegraphics[width=4cm]{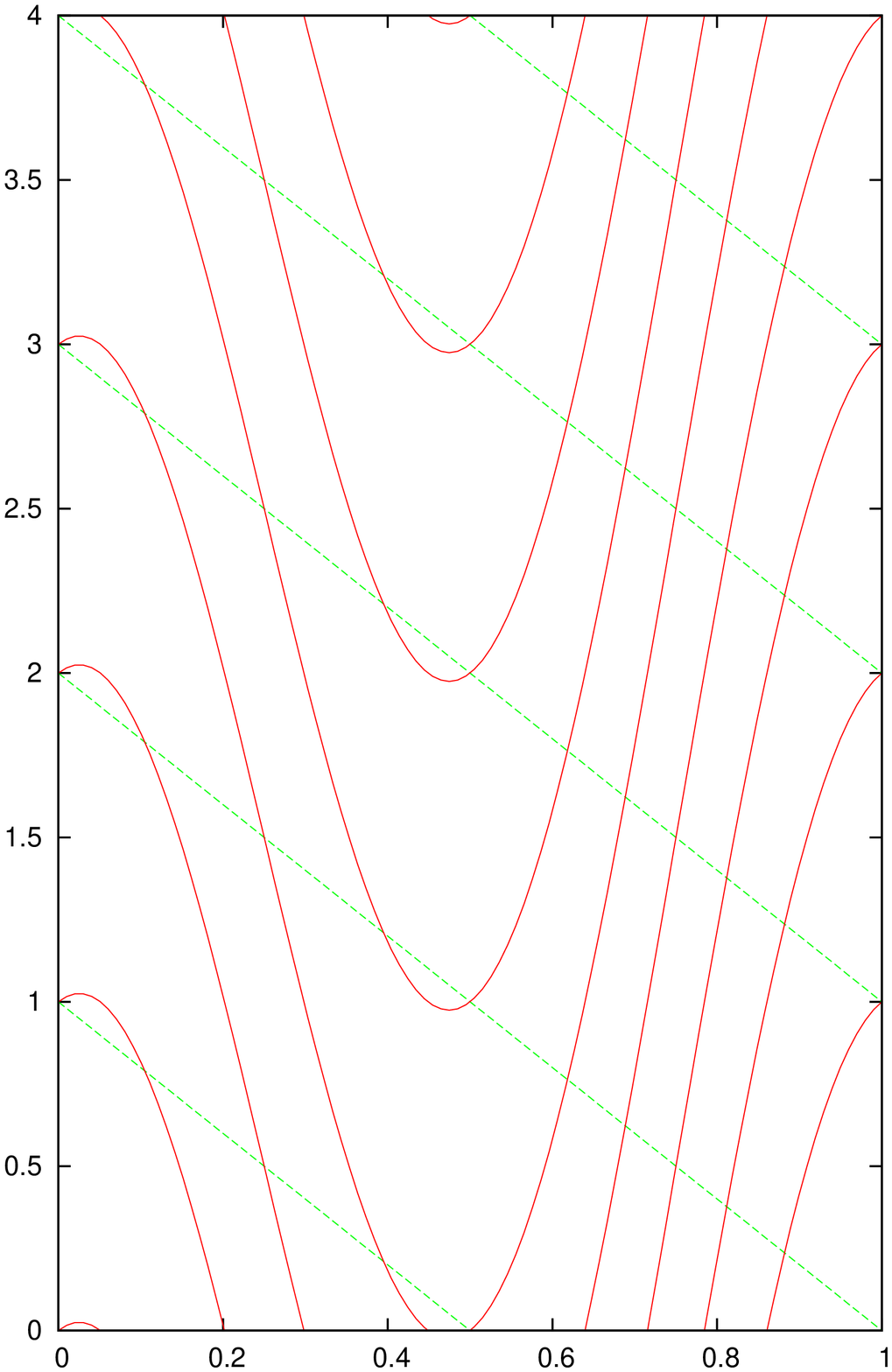} 
    \caption{On the left the reference picture for $A=0$, on the right the situation for $A>0$}
    \label{shape}
    \end{center}

\end{figure}
As $A$ increases, such line deforms and resembles the shape of $\dot\phi$, as we can observe in figure \ref{shape}. As we noticed in proposition \ref{period2}, 2-orbits obtained by intersecting $\ell$ and $\ell^1$ are $(+)$-orbits; we now claim that the highest energy point $(t_2,v_2)$ of the orbit lies where $\dot\phi>0$. In fact we know that $\dot\phi(t_2)$ is the velocity gained on the collision at time $t_1$; since we want that $v_2>v_1$ we need such quantity to be positive. Having that fixed, $\nu_2=\ddot\phi(t_1)T'(v_2)=\ddot\phi(t_2)T'(v_2)$ and $v_1<v_2$, therefore we have that condition \eqref{sufficient} is satisfied if $(t_2,v_2)\in\ell\cap\ell^1$ belongs to this set:
\[
\tilde{E}^+_{c_1'c_2'}\defeq\left\{(t,v)\st \frac{c_1'}{T'(v)}<\ddot\phi(t)<\frac{c_2'}{T'(v)}, \dot\phi(t)>0 \right\}.
\]
First notice that this set is an $\bigo{1/(A\cdot T'(v))}$-thin strip that lies $\bigo{1/(A\cdot T'(v))}$ on the right of the vertical line $\tau=0$ (that corresponds to $c_2'=-1$). By direct inspection we obtain that in $(t,T)$ coordinates, each branch of $\ell^0$ is a straight line with fixed angular coefficient -2 and each branch of $\ell^1$ near $\tau=0$ is approximated by a parabola that intersects $\tau=0$ with positive derivative (close to 2); the maximum of such parabola is given by the equation:
\[
\ddot\phi(\tau)=-\frac{1}{T'(v_n(\tau))}<-\frac{1}{T'(v)}<\frac{c_1}{T'(v)}.
\]
Figure \ref{ellipticwindow} illustrates the properties we just described.
\begin{figure}[!h] 
\begin{center} 
\includegraphics[width=6cm]{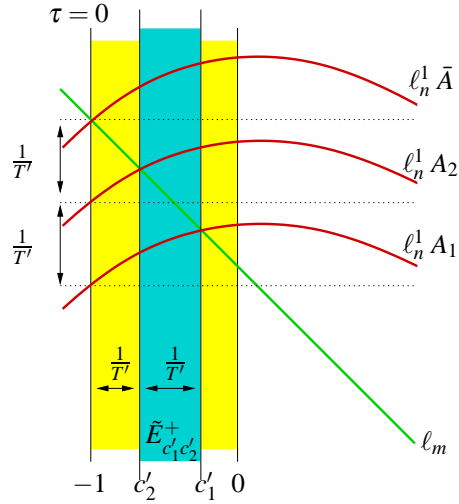}
\end{center} 
\caption{Explicit construction, in $(t,T)$ coordinates, of values of the parameter $A$ for which we have an elliptic periodic point of period 2 with given bounds on the multiplier.}
\label{ellipticwindow}
\end{figure}
The key fact to notice is that the values of $A$ we are seeking are close to values $\bar A$ of the parameter for which the intersection lies on the vertical line $\tau=0$. Let us compute the intersection of $\ell_m$ and $\ell^1_n$ with the vertical line $\tau=0$:
\begin{equation}
\ell_m(0)=(0,v(-1/2+m))
\label{int1}
\end{equation}
\begin{equation}
\ell^1_n(0)=(0,v(1/2+n)+2 A)
\label{int2}
\end{equation}
Therefore if we want $\ell$ and $\ell^1$ to intersect on the line $\tau=0$ we need to find $\bar A$ such that points in \eqref{int1} and \eqref{int2} are equal for some $n,m$, i.e.: 
\begin{equation}
v\left(n+\frac{1}{2}\right)+2\bar A=v\left(m-\frac{1}{2}\right)
\end{equation}
Now it is clear that we can find $A_1$ and $A_2$ as in the picture such that the intersection lies on the boundary of \smash{$\tilde{E}^+_{c_1'c_2'}$}. Using the properties we described above it is also clear that, in $(t,T)$ coordinates, the distances between intercepts of $\ell^1$ corresponding to each $A_i$ with the vertical $\tau=0$ are linear functions of the $t$ coordinates of the intersections themselves, therefore of order $\bigo{1/(A\cdot T'(v))}$.
More precisely, mimicking equation \eqref{int2} and recalling that $T'\sim T^{1-1/\gamma}$, we obtain that there exist $C_1$ and $C_2$ such that if    
\[ 
v\left(m-\frac{1}{2}-\frac{C_2}{A}\frac{1}{m^{1-1/\gamma}}\right)<v\left(n+\frac{1}{2}\right)+2A<v\left(m-\frac{1}{2}-\frac{C_1}{A}\frac{1}{m^{1-1/\gamma}}\right).\] then the intersection $(\ell_m\cap\ell^1_n)\cap\tilde{E}^+_{c_1'c_2'}\not = \emptyset$.
\end{proof}
The condition we obtained is basically an arithmetic condition on $A$ and $\gamma$. In the next subsection we prove that this condition is satisfied for parameters $A$ as in the statement of Theorem B. At that point we will be only left with checking the non-degeneracy condition. 
\subsection{Arithmetic condition}
\label{arith}
In this section we are going to prove an arithmetic result that is of independent interest; for simplicity we state the arithmetic condition in a slightly simplified form with respect to the case in consideration. Namely we drop the $1/2$ that appears in the statement of proposition \ref{condition} and we reverse the signs of $C_1$ and $C_2$. One can easily verify that this does not affect the proof in any sense.\\
The condition is reminiscent of the Khinchin's theorem on Diophantine approximation \cite{khin}. In fact we want to investigate parameters $\gamma$ and $a$ such that the following inclusion is true for infinitely many $n$ and $m \in\naturals$: 
\[
\left(n^{1/\gamma}+a\right)^\gamma\in\left(m+\frac{C_1}{a}m^{-\xi},m+\frac{C_2}{a}m^{-\xi}\right),
\]
for an appropriate (and fixed) choice of $C_2>C_1>0$ and $\xi>0$. In our case $a=2A$ and $\xi=1-1/\gamma$.\\
Let us first introduce some useful definitions:
\begin{definition}
Let us fix $\xi>0$, $\gamma>1$, $C_2>C_1>0$. Then 
\[
\mathscr{G}_{a,m}\defeq\left(m+\frac{C_1}{a}m^{-\xi},m+\frac{C_2}{a}m^{-\xi}\right);
\]
\[
\mathscr{G}_a\defeq\bigcup_{m\in\naturals}\mathscr{G}_{a,m}\quad
\gs_a\defeq\mathscr{G}_a^{1/\gamma};
\]
\[
X_a\defeq\left\{n^{1/\gamma}+a,\ n\in\naturals\right\}.
\]
\end{definition}
Using this notation a parameter $a$ satisfies the arithmetic condition if the cardinality $\left|X_a\cap\gs_a\right|$ is infinite.
\begin{definition} Let $n,k\in\naturals$:
\[
\A\defeq\left\{a\in\reals^+\st \left|X_a\cap\gs_a\right|=\infty \right\};
\]
\[
\An{n}\defeq\left\{a\in\reals^+\st \left(n^{1/\gamma}+a\right)^\gamma\in\mathscr{G}_a \right\};
\]
\[
\Ank{n}{k}\defeq\left\{a\in\reals^+\st \left(n^{1/\gamma}+a\right)^\gamma\in \mathscr{G}_{a,(n+k)}\right\};
\]
\end{definition}
Clearly $\An{n}=\bigcup_k\Ank{n}{k}$, moreover if $\tilde{\A}^{n_0}\defeq \bigcup_{n\geq n_0}\An{n}$, then $\A=\bigcap_{n_0}\tilde{\A}^{n_0}=\limsup_{n\to\infty}\An{n}$
\begin{lemma} For all real $\gamma>1$ $\A$ is a residual set in $\reals^+$.
\label{topological}
\end{lemma}
\begin{proof}
Each $\tilde{\A}^{n_0}$ is open since it is a union of open sets. Moreover it is dense since the distance between endpoints of consecutive intervals belonging to $\gs_a$ goes to 0 as $m\to\infty$, and so do the distances between endpoints of the intervals belonging to $\An{n}$ as $n\to\infty$. As the point 0 is the limit point of  the left endpoints of the first interval in $\An{n}$ we conclude that $\tilde{\A}^{n_0}$ is dense in $\reals^+$.
\end{proof}
Let us define the following conditions involving $\xi$ and $\gamma$:
\[
 \xi \leq 1 \quad \textrm{(diverging)}; 
\]
\[
\xi<\frac{1}{\gamma}\quad \textrm{(overlapping)}.
\]
Notice that since $\gamma>1$ the overlapping condition implies the diverging condition.
Now we can state the result as follows:
\begin{theorem}
If the diverging condition does not hold, then $\A$ has measure 0.
If the diverging condition holds then $\A$ has full measure in $\reals^+$; moreover if the overlapping condition holds as well then $\A$ is the whole $\reals^+$.
\end{theorem}
Notice that, as in Khinchin's theorem, we obtain that the required property is satisfied either by a null set or by a full measure set. This dichotomy seems to be quite common in approximation problems similar to the one we are studying.  
\begin{proof}
The proof will be presented in four steps.
\paragraph{Step one} \emph{By restricting to a compact interval of possible parameters we get rid of the parameter $a$ in the definition of the sets $\mathscr{G}$}\\
We restrict ourselves to a bounded interval of parameters $a\in[\alpha,\beta]$ with the condition $\beta/\alpha<C_2/C_1$. This condition implies the middle inequality in the following expression:
\[
\frac{C_1}{\beta}<\frac{C_1}{a}<\frac{C_1}{\alpha}<\frac{C_2}{\beta}<\frac{C_2}{a}<\frac{C_2}{\alpha}.
\]
We can therefore build a superset $\gs_{\alpha\beta}^*$ and a subset $\gs_{\alpha\beta*}$ such that $\gs_{\alpha\beta}^*\supset\gs_a\supset\gs_{\alpha\beta*}$ for any $a\in[\alpha,\beta]$, therefore proving the result using either $\gs^*_{\alpha\beta}$ (to obtain estimates from above) or $\gs_{\alpha\beta*}$ (to obtain estimates from below) provides it for the original problem as well. Therefore we now fix once and for all $0<\alpha<\beta<(C_2/C_1)\alpha$ and we prove the result for parameters $a$ in such interval; to avoid the introduction of new notation we redefine $\A$, $\An{n}$, $\Ank{n}{k}$ as their intersection with the interval $[\alpha,\beta]$. As we can write $\reals^+$ as a countable union of such intervals we obtain the desired result. 
\paragraph{Step two} \emph{We investigate the structure of the sets $\An{n}$}\\
\begin{figure}[!ht]
\includegraphics[height=4cm]{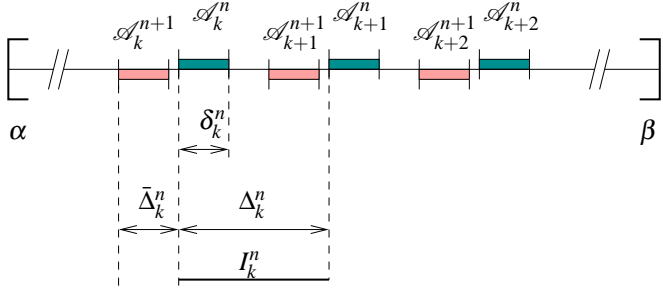}
\caption{Definition of  $\delta^n_k,\bar\Delta^n_k,\Delta^n_k,I^n_k$.}
\label{definition}
\end{figure}
Define $\delta^n_k,\bar\Delta^n_k,\Delta^n_k,I^n_k$ as in figure \ref{definition}. The following lemma provides some useful estimates:
\begin{lemma}
Define the following positive quantities:
\[
\ell_+\defeq \frac{C_2}{\alpha}-\frac{C_1}{\beta}\quad \ell_-\defeq \frac{C_2}{\beta}-\frac{C_1}{\alpha},
\] and the following set $K_n\defeq\left\{k\in\naturals\st\Ank{n}{k}\not = \emptyset\right\}$. Then:
\begin{equation}
\tag{E1}
\delta^n_k=\frac{\ell}{\gamma}(n+k)^{-\xi-(1-1/\gamma)}+\hot\textrm { for $\ell\in(\ell^+,\ell^-)$}
\label{e1}
\end{equation}
\begin{equation}
\tag{E2}
\Delta^n_k=\frac{1}{\gamma}(n+k)^{-\left(1-1/\gamma\right)}+\hot
\label{e2}
\end{equation}
\begin{equation}
\tag{E3}
\bar{\Delta}^n_k=\frac{1}{\gamma}\left(1-\frac{1}{\gamma}\right)k n^{1/\gamma-2}+\hot
\label{e3}
\end{equation}
\begin{equation}
\tag{E4}
K_n\sim[\gamma\alpha\cdot n^{1-1/\gamma}+\hot,\gamma\beta\cdot n^{1-1/\gamma}+\hot]\cap\naturals
\label{e4}
\end{equation} 
\end{lemma}
\begin{proof} 
We first bound the length of the intervals $\Ank{n}{k}$:
\begin{eqnarray*}
\delta^n_k&<&\left(n+k+\frac{C_2}{\alpha}(n+k)^{-\xi}\right)^{1/\gamma}-\left(n+k+\frac{C_1}{\beta}(n+k)^{-\xi}\right)^{1/\gamma}\\
%&=&\left(n+k\right)^{1/\gamma}\left(\left(1+\frac{C_2}{\alpha}\left(n+k\right)^{-\xi-1}\right)^{1/\gamma}-\left(1+\frac{C_1}{\beta}\left(n+k\right)^{-\xi-1}\right)^{1/\gamma}\right)\\
&=&\frac{\ell^+}{\gamma}\left(n+k\right)^{-\xi-(1-1/\gamma)}+\hot
\end{eqnarray*}
The bound from below is similar and yields the expected result. 
Then we estimate the length of the intervals $I^{n}_{k}$:
\begin{eqnarray*}
\Delta^n_k&=&\left(n+k+1\right)^{1/\gamma}-n^{1/\gamma}-\left(\left(n+k\right)^{1/\gamma}-n^{1/\gamma}\right)\\
&=&\left(n+k+1\right)^{1/\gamma}-\left(n+k\right)^{1/\gamma}\\
&=&\frac{1}{\gamma}\left(n+k\right)^{1/\gamma-1}+\hot
\end{eqnarray*}
Next the offset of two subsequent $\Ank{n}{k}$:
\begin{eqnarray*}
\bar\Delta^n_k&=&\left(n+k+1\right)^{1/\gamma}-\left(n+1\right)^{1/\gamma}-\left(\left(n+k\right)^{1/\gamma}-n^{1/\gamma}\right)\\
&=&\left(n+k+1\right)^{1/\gamma}-\left(n+k\right)^{1/\gamma}-\left(\left(n+1\right)^{1/\gamma}-n^{1/\gamma}\right)\\
&=&\frac{1}{\gamma}\left(\left(n+k\right)^{1/\gamma-1}-n^{1/\gamma-1}\right)+\hot\\
&=&-\frac{1}{\gamma}\left(1-\frac{1}{\gamma}\right)k\cdot n^{1/\gamma-2}+\hot
\end{eqnarray*}
Finally we estimate $K_n$:
\begin{eqnarray*}
K_n&=&\left[\left(n^{1/\gamma}+\alpha\right)^\gamma-n,\left(n^{1/\gamma}+\beta\right)^\gamma-n\right]\cap\naturals\\
&=&\left[\gamma\alpha\cdot n^{1-1/\gamma}+\hot,\gamma\beta\cdot n^{1-1/\gamma}+\hot\right]\cap\naturals;
\end{eqnarray*}
notice that 
\[ |K_n|\sim\gamma n^{1-1/\gamma}\cdot(\beta-\alpha).\]
\end{proof}
\paragraph{Step three} \emph{Overlapping regime}\\
From the previous estimates we can already obtain the result in the overlapping regime. In fact \eqref{e4} implies that $k$ is $\mathscr{O}\left(n^{1-1/\gamma}\right)$, therefore, by \eqref{e3}, $\bar\Delta^n_k$ is $\mathscr{O}\left(n^{-1}\right)$. This implies that if $-\xi-(1-1/\gamma)>-1$ (i.e. $\xi<1/\gamma$ that is the overlapping condition) $\Ank{n}{k}$ and $\Ank{n+1}{k}$ will eventually overlap as, by \eqref{e1}, the length of the intervals $\Ank{n}{k}$ goes to zero slower than their offset. Since they overlap and they are moving like $1/n$, they will eventually get out to the left so for each fixed $k$ they are going to cover the whole interval, therefore $\An{n}$ is going to cover $[\alpha,\beta]$ infinitely many times, and $\A$ will contain $[\alpha,\beta]$.\\
\paragraph {Step four} \emph{Non-overlapping regime}\\
We will now focus on the strictly non-overlapping regime i.e. $\xi>1/\gamma$; the critical case $\xi=1/\gamma$ will be considered later as it is just a combination of other cases.\\ 
Define, for any Borel set $E\subset[\alpha,\beta]$, $\pr(E)=\leb(E)/(\beta-\alpha)$ (the normalized Lebesgue measure) as a probability measure on $[\alpha,\beta]$.
We are going to prove that the set $\A$ has either full measure or measure zero using the following strong form of the Borel-Cantelli lemma:
\begin{lemma}[Borel-Cantelli-Erd\"os-R\'enyi \cite{Erdos}]
Let $\{\A_k\}$ be a sequence of events on a probability space $(\Omega,\mathscr{F},\pr)$. If
\[
\sum_{n=1}^\infty\pr(\A_n)<\infty\quad\textrm{(convergence)},
\]
then $\pr(\limsup \A_n)=0$.
If instead
\[
\sum_{n=1}^\infty\pr(\A_n)=\infty\quad\textrm{(divergence)}
\]
and 
\[
\liminf\frac{\sum_{k,l=1}^n\pr(\A_k\A_l)}{\left(\sum_{k=1}^n\pr(\A_k)\right)^2}=1\quad\textrm{(weak independence)},
\]
then $\pr(\limsup \A_n)=1.$
\end{lemma}
We will get the result by verifying the above conditions for the sets $\An{n}$. The following two lemmas deal with the estimates of the convergence/divergence condition and the weak independence condition respectively.
\begin{lemma}
Fix $\bar n$ and let $N$ tend to infinity. Then:
\begin{equation}
\sum_{n=\bar{n}}^{N}\pr\left(\An{n}\right)=\frac{1}{1-\xi}\ell N^{1-\xi}+\hot
\label{S}
\end{equation}
\end{lemma}
\begin{proof}
Let us first compute:
\begin{eqnarray*}
\leb\left(\An{n}\right)&=&\sum_{k\in K_n}\leb\left(\Ank{n}{k}\right)=\sum_{k\in K_n}\delta^n_k=\frac{\ell}{\gamma}\sum_{k\in K_n}(n+k)^{-\xi-1+1/\gamma}\\
&=&\frac{\ell}{\gamma}\sum_{k\in K_n}n^{-\xi-1+1/\gamma}+\hot=(\beta-\alpha)\ell n^{-\xi}+\hot\\
\end{eqnarray*}
Then by normalizing and summing on $n$ we get \eqref{S}:
\[
\sum_{n=\bar n}^N\pr\left(\An{n}\right)=\ell \sum_{n=\bar n}^N n^{-\xi}+\hot=\frac{1}{1-\xi}\ell N^{1-\xi}+\hot
\]
\end{proof}
Observe that the diverging condition implies that \eqref{S} diverges as $N\to\infty$, while the series converges if the condition is not satisfied.\\
\begin{lemma}
Fix $\bar n$, let $N$ go to infinity, then:
\[
\sum_{n=\bar n}^N \sum_{m=\bar n}^N \pr\left(\An{n}\An{m}\right)=\ell^2\frac{1}{(1-\xi)^2}N^{2-2\xi}+\hot
\]
\end{lemma}
\begin{proof}
First of all, by symmetry we can consider just $m>n$ by paying a factor of 2 and some diagonal higher order terms; then we separate again $\An{n}$ in their respective components $\Ank{n}{k}$ 
\[
\sum_{n=\bar n}^N \sum_{m=\bar n}^N \pr\left(\An{n}\An{m}\right)=2 \sum_{n=\bar n}^N \sum_{m=n}^N \pr\left(\An{n}\An{m}\right)+\hot=
\]
\[
=2 \sum_{n=\bar n}^N\sum_{k\in K_n}\sum_{m= n}^N\sum_{l\in K_m} \pr\left(\Ank{n}{k}\Ank{m}{l}\right)
\]
\begin{figure}[!b]
\includegraphics[height=3cm]{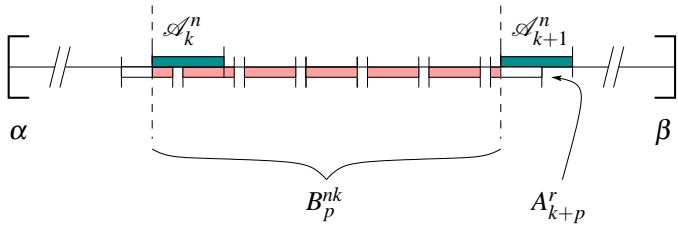}
\caption{Relation between \emph{waves} $B_p^{nk}$ and sets $\Ank{n}{k}$}\label{waves}
\end{figure}
Now let us focus on the last two sums, fixing momentarily $n$ and $k$. We can write the following equality:
\begin{equation}
\sum_{m= n}^N\sum_{l\in K_m} \pr\left(\Ank{n}{k}\Ank{m}{l}\right)=\sum_{p=1}^{P^n_k}\pr\left(\Ank{n}{k}B^{nk}_p\right)+\textrm{error term}
\label{AB}
\end{equation} 
where we define the sets $B^{nk}_p$ as the $p$-th \emph{wave} (see figure \ref{waves}):
\[B^{nk}_p\defeq\bigsqcup_{r\in\naturals}\Ank{r}{k+p}\cap I^n_k.
\]
$P^n_k$ is an appropriate number that is estimated by the next lemma and the error term is due to the fact that the last wave could be incomplete.
\begin{lemma}
The following estimate holds:
\[
P^n_k=\left(\frac{N^{1-{1/\gamma}}}{n^{1-{1/\gamma}}}-1\right)\cdot k+\bigo{1}.\]
\end{lemma}
\begin{proof}
As it follows from \eqref{e3}, the offset between $\Ank{r}{k+p}$ and $\Ank{r+1}{k+p}$ is $\bar\Delta^r_{k+p}$. $B^{nk}_p$ is the union over all $r$ such that $\Ank{r}{k+p}$ intersects $I^n_k$; for each fixed $p$ there is a wave of $\Ank{r}{k+p}$ that will be inside $I^n_k$ for some time and then leave the set. $P^n_k$ is the number of waves that will pass through $I^n_k$ in the time $N$. This means that:
\[
\left(N+k+P^n_k\right)^{1/\gamma}-N^{1/\gamma}\sim(n+k)^{1/\gamma}-n^{1/\gamma},
\]
that is:
\[
N^{1/\gamma-1}\left(k+P^n_k\right)\sim n^{1/\gamma-1}k\ \Rightarrow\ P^n_k=\left(\frac{N^{1-{1/\gamma}}}{n^{1-{1/\gamma}}}-1\right)\cdot k+\bigo{1}.
\]
\end{proof}
The error term in \eqref{AB} can be easily bounded by the order of $\pr\left(\Ank{n}{k}\right)$, as we miscount of at most 1 wave. 
We are left with the computation of $\pr\left(\Ank{n}{k}B^{nk}_p\right)$. Each $B^{nk}_p$ is the union of intervals that are $\bar\Delta^r_{k+p}$ apart and $\delta^r_{k+p}$ long. The ratio of such data gives the portion of the interval covered by each wave. 
\begin{lemma}
Let us introduce the parameter $\eta=(k+p)/k$. Then we have:
\[
\frac{\delta^r_{k+p}}{\bar\Delta^r_{k+p}}=
\ell\frac{\gamma}{\gamma-1}\frac{n^{1-\xi}}{k}\eta^\frac{1-\xi\gamma}{\gamma-1}+\hot\defeq\lambda^n_k(p).
\]
\label{ratio}
\end{lemma}
\begin{proof}
By definitions of $\delta_k^n$ and $\bar\Delta^n_k$ we have:
\[
\frac{\delta^r_{k+p}}{\bar\Delta^r_{k+p}}=\ell\frac{\gamma}{\gamma-1}\frac{(r+k+p)^{-\xi-1+1/\gamma}}{(k+p)(r^{1/\gamma-2})}.
\]
We need an estimate on $r$: acting as before for the computation of $P$ we get:
\[
r^{1/\gamma-1}\left(k+p\right)\sim n^{1/\gamma-1}k\ \Rightarrow\ r=n\left(\frac{k+p}{k}\right)^\frac{\gamma}{\gamma-1}.
\]
 We rewrite the previous expression as:
\[
\ell\frac{\gamma}{\gamma-1}\frac{(n\eta^{\gamma/(\gamma-1)}+\eta k)^{-\xi-1+1/\gamma}}{\eta k\cdot n^{1/\gamma-2}\cdot \eta^{(\gamma/(\gamma-1))(1/\gamma-2)}},
\]
that is:
\[
\ell\frac{\gamma}{\gamma-1}\frac{\left(\left(n+\eta^{-1/(\gamma-1)}k\right)\eta^{\gamma/(\gamma-1)}\right)^{-\xi-1+1/\gamma} }{k\cdot n^{1/\gamma-2}\cdot\eta^{-\gamma/(\gamma-1)}}+\hot=\ell\frac{\gamma}{\gamma-1}\frac{n^{1-\xi}}{k}\eta^\frac{1-\xi\gamma}{\gamma-1}+\hot
\]
\end{proof}
Therefore $\pr\left(\Ank{n}{k}B^{nk}_p\right)=\lambda^n_k(p)\cdot \pr(\Ank{n}{k})+\mathscr{O}(\bar\Delta^r_{k+p}/(\beta-\alpha))$, where the error term comes from the granularity of the set $B$. Therefore:
\begin{eqnarray*}
\sum_{p=1}^{P^n_k}\pr\left(\Ank{n}{k}B^{nk}_p\right)&=& \pr(\Ank{n}{k})\sum_{p=1}^{P^n_k}\lambda^n_k(p)+\hot\\
&=& \pr(\Ank{n}{k})\ell\frac{\gamma}{\gamma-1}\frac{n^{1-\xi}}{k}\sum_{p=1}^{P^n_k}\left(\frac{k+p}{k}\right)^{\frac{1-\xi\gamma}{\gamma-1}}+\hot\\
&=& \pr(\Ank{n}{k})\ell\frac{\gamma}{\gamma-1}\frac{n^{1-\xi}}{k^{\frac{(1-\xi)\gamma}{\gamma-1}}}\sum_{p=1}^{P^n_k}(k+p)^{\frac{1-\xi\gamma}{\gamma-1}}+\hot\\
&=& \pr(\Ank{n}{k})\ell\frac{1}{1-\xi}\frac{n^{1-\xi}}{k^{\frac{(1-\xi)\gamma}{\gamma-1}}}\left(\left(\left(k+P^n_k\right)^{\frac{(1-\xi)\gamma}{\gamma-1}}\right)-k^{\frac{(1-\xi)\gamma}{\gamma-1}}\right)+\hot\\
&=& \pr(\Ank{n}{k})\ell\frac{1}{1-\xi}\frac{n^{1-\xi}}{k^{\frac{(1-\xi)\gamma}{\gamma-1}}}k^{\frac{(1-\xi)\gamma}{\gamma-1}}\left(\left(\frac{N^{1-{1/\gamma}}}{n^{1-1/\gamma}}\right)^{\frac{(1-\xi)\gamma}{\gamma-1}}-1\right)+\hot\\
&=& \pr(\Ank{n}{k})\ell\frac{1}{1-\xi}n^{1-\xi}\left(\frac{N^{1-\xi}}{n^{1-\xi}}-1\right)+\hot
\end{eqnarray*}
Now we sum over $n$ and $k$:
\begin{eqnarray*}
\sum_{n=\bar n}^N\sum_{k\in K_n}\pr(\Ank{n}{k})\ell\frac{1}{1-\xi} n^{1-\xi}\left(\frac{N^{1-\xi}}{n^{1-\xi}}-1\right)+\hot=\\
\sum_{n=\bar n}^N\ell\frac{1}{1-\xi} n^{1-\xi}\left(\frac{N^{1-\xi}}{n^{1-\xi}}-1\right)\frac{1}{\beta-\alpha}\sum_{k\in K_n}\delta^n_k+\hot=\\
\sum_{n=\bar n}^N\ell^2\frac{1}{1-\xi} n^{1-2\xi}\left(\frac{N^{1-\xi}}{n^{1-\xi}}-1\right)+\hot=\\
\ell^2\frac{1}{1-\xi}\sum_{n=\bar n}^N\left(N^{1-\xi}n^{-\xi}-n^{1-2\xi}\right)+\hot=\\
\ell^2\frac{1}{2(1-\xi)^2}N^{2-2\xi}+\hot
\end{eqnarray*}
Recalling the factor 2 we had in the beginning of the estimate we get the desired result. 
\end{proof}
The last two lemmas prove that the weak independence condition is always satisfied independently of $\xi$. Therefore we have only to check the diverging condition. In the diverging regime we can conclude that the set $\A$ has full measure, as we found a subset of full measure; in the non-diverging regime we can as well conclude that $\A$ has zero measure, as it is contained in a zero-measure set.
\paragraph{Step five} \emph{The critical case $\xi=1/\gamma$}\\
For $\xi=1/\gamma$ we have that the overlapping condition is satisfied for small enough $a$, as $\delta^n_k$ grows bigger as $a$ decreases. As we notice from lemma \ref{ratio} we can find a critical $\bar a$ such that for $a<\bar a$ we have overlapping and for $a>\bar a$ we have no overlapping. 
\end{proof}
In the case in our consideration $\xi=1-1/\gamma$; this implies that the diverging condition is always satisfied and so we proved Theorem B up to the non-degeneracy condition. 
\begin{remark}
For this particular value of $\xi$ the overlapping condition reads $\gamma<2$; the critical case is therefore $\gamma=2$. In such case it is actually possible to explicitly compute a value $A\not\in\A$, whereas in general for $\gamma>2$ we still ignore if $\reals^+\setminus\A\not=\emptyset$.
\end{remark}
\begin{remark}
The technique we developed can be applied to $(-)$-orbits as well. The arithmetic condition relative to such orbits turns out to be more restrictive than the one for $(+)$-orbits; more precisely we get $\xi=2-2/\gamma$. This implies that the diverging condition is not anymore guaranteed. In fact it fails for $\gamma>2$, which means that such orbits appear for arbitrarily high energies for almost all sinusoidal motions only for $\gamma<2$.\\
Having studied all possible 2-periodic orbits, we notice how the conditions we stated are actually also necessary conditions for the presence of elliptic 2-periodic orbits. This implies the following interesting results:
\begin{itemize}
\item by the previous remark, for $\gamma=2$ there is at least a value of $A$ such that the system has only finitely many elliptic islands of period 2.
\item if $\gamma>2$ we have infinitely many $(-)$-elliptic islands only for a null-measure set. Notice however that lemma \ref{topological} does not depend on $\xi$, therefore such set is non-empty.  
\end{itemize} 
\label{fico}
\end{remark} 
Remark \ref{fico} allows also to prove the following proposition regarding the total measure of elliptic islands, given the following estimate of the measure of a single elliptic island: the measure of each $(+)$-elliptic island is of order $1/T'^3$. We are going to prove such estimate in proposition \ref{measureElliptic}.
\begin{proposition}
If $\gamma>4/3$ the total Lebesgue measure of $(+)$-elliptic islands of period 2 is finite. If $\gamma<4/3$ the total Lebesgue measure of elliptic islands is infinite.   
\label{islandmeasure}
\end{proposition}
\begin{proof}
We can obtain a rough upper bound to the total measure of elliptic island by summing the measure of a single island over all possible islands that can appear. As there could be one for each $T\sim n+1/2$ we have the following estimate:
\[
\textrm{Leb(islands)}<\textrm{Const}\cdot\sum_{n=1}^{\infty}\frac{1}{T_n'^3}=\textrm{Const}\cdot\sum_{n=1}^{\infty}n^{-3(1-1/\gamma)}.
\] 
The series converges for $\gamma>3/2$. To sharpen the estimate we need to take into account that for some of the $n$ we do not have an elliptic island. For $\gamma<2$ this may be estimated quite easily, as we have the following expansion:
\[
(n^\frac{1}{\gamma}+A)^\gamma=n+\gamma A n^{1-1/\gamma}+o(1).
\]
From the previous expansion it is clear that: 
\[
\left\{(n^\frac{1}{\gamma}+A)^\gamma\right\}= \left\{\gamma A n^{1-\frac{1}{\gamma}}\right\}+o(1).
\]
such function has an infinite number of branches, let us denote such branches by $k$. Each branch will start at $n_k\sim k^{\gamma/(\gamma-1)}$. The arithmetic condition can be expressed in terms of $k$ in the following way:
\[ 
\left\{n^{1-\frac{1}{\gamma}}\right\}< \bigo{k^{-1}}.
\]
Figure \ref{arithmetic} pictures this condition; for each branch $k$ we have elliptic islands until the fractional part grows too large and the arithmetic condition no longer holds true.
\begin{figure}[!hb]
\includegraphics[width=6cm,angle=270]{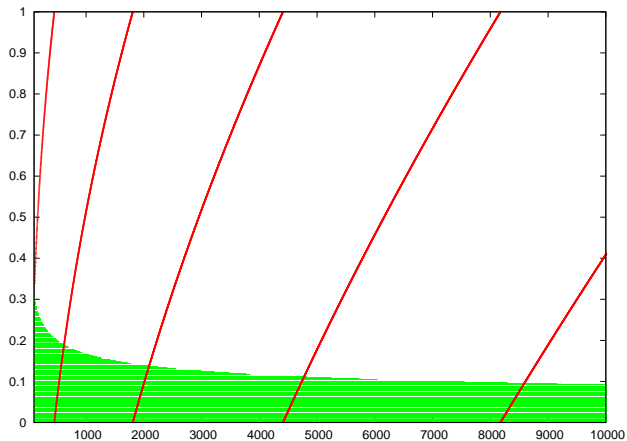}
\caption{The arithmetic condition for $\gamma<2$} 
\label{arithmetic}
\end{figure} 
Given this fact it is easy to estimate the number of islands belonging to the $k$-th branch. We compute the derivative of $n^{1-\frac{1}{\gamma}}$ for $n_{k+1}$, obtaining a linear lower bound on the growth of $n^{1-\frac{1}{\gamma}}$ in the $k$-th branch: 
\[
\left\{n^{1-\frac{1}{\gamma}}\right\}> n_{k+1}^{-\frac{1}{\gamma}}\cdot(n-n_k) \sim (k+1)^{-\frac{1}{\gamma-1}}\cdot(n-n_k).
\] 
The smallest $n$ for which the arithmetic condition fails can therefore be bound from above by requiring:
\[
(k+1)^{-\frac{1}{\gamma-1}}\cdot(n-n_k)<(k+1)^{-1} \textrm{ i.e. } n-n_k <(k+1)^{\frac{2-\gamma}{\gamma-1}}.
\]
therefore for the $k$-th branch we get at most $\bigo{k^{\frac{2-\gamma}{\gamma-1}}}$ elliptic islands. We now multiply this number by the measure of such islands and sum over all branches $k$ to get the total measure:
\[
\textrm{Leb(islands)}<C\cdot\sum_{k=1}^{\infty}k^{\frac{2-\gamma}{\gamma-1}}k^{-3},
\]
that converges for $\gamma>4/3$. Notice that along the same lines we can get a lower bound of the same order, that means that the total measure of islands diverge for $\gamma\leq 4/3$.
\end{proof} 
\subsection{Non-degeneracy condition}
According to general KAM theory, there exists a stability island around each point of a periodic orbit provided that generic non-resonance and non-degeneracy conditions are satisfied. Following \cite{laz}:
\begin{definition}
An elliptic fixed point $p$ of a two-dimensional symplectic diffeomorphism $f$ is said to be \emph{general elliptic} if:
\begin{itemize}
\item the multiplier $\lambda_p$ is such that $\lambda_p^k\not =1$ for $k=1,2,3,4$ (non-resonance up to order 4);
\item the Birkhoff normal form is non-degenerate, i.e. a quantity that can be written in terms of derivatives up to fourth order is different from zero (see below).
\end{itemize}
\end{definition}
\begin{theorem}[KAM]
If $p$ is general elliptic, then it is stable, i.e. for each neighbourhood $U$ of $p$ there exist another neighbourhood $V$ such that $\fa k,\ F^k(V)\subset U$.
\end{theorem}
Stability around the point implies the presence of an elliptic island. As noted before, the construction we described yields elliptic points with multiplier which can be chosen to belong to some prescribed interval; this implies that we can a priori avoid resonances. The non-degeneracy condition can be explicitly computed by the following procedure:
\begin{itemize}
\item We perform a linear change of coordinates such that the differential of the map $\deh F$ at the fixed point is a rotation in the new coordinate $(\xi,\eta)$ where $(\xi=0,\eta=0)\mapsto p$.
\item We compute the Taylor expansion coefficients up to order four (excluded) in the coordinates $u=\xi+i\eta$ and $\bar u$ obtaining the following expression:
\[
u\mapsto \lambda_pu+A_3u^2+A_4u\bar u+A_5 {\bar u}^2+A_6 u^3+A_7 u^2 \bar u+A_8 u{\bar u}^2+ A_9 {\bar u}^3+\bigo{4}.
\]
\item We compute the following expression:
\[
\omega=-i\left\{i \Im(\bar \lambda_p A_7)+3|A_3|^2\frac{\lambda_p+1}{\lambda_p-1}+|A_5|^2\frac{\lambda_p^3+1}{\lambda_p^3-1} \right\}.
\]
\end{itemize}
The non-degeneracy condition requires that $\omega\not =0$. The coefficients $A_i$ contain derivatives of $T$ up to order 3 and derivatives of $\phi$ up to order 4. As for high energies we have $T'_i\gg 1$, instead of computing all $A_i$ exactly, we perform an expansion in terms of powers of $T'_i$ and compute the highest order non-zero term, taking into account that ellipticity implies $\ddot\phi T_i'=\nu_i\sim 1$ (i.e. condition \eqref{sufficient}). We find by direct computation\footnote{Computations were made using the software Mathematica. A printout can be found at \url{http://www.math.umd.edu/~jacopods/bnf.pdf}} that the highest nonzero term in $\omega$ is of order ${T_i'}^3$. As a further simplification we note that we have $T'_2=T'_1+\bigo{T''}$; this implies that if we compute $\omega$ by setting $T_1'=T_2'\defeq T'$ and find a quantity bounded away from zero in this limit, it will be bounded away from zero also for all $T_i'$ sufficiently large.
The coefficient of order 3 turns out to be the following polynomial in $\nu=\ddot\phi T'\in (-1,0)$:
\[
\omega_3=\frac{(2+\nu){\dot{\ddot{\phi}}}^2}{64{D^6}/{T'^3}}\left[2(\nu^2+4\nu+6)+\left(-i\nu(2+\nu)^2\left(3\frac{\lambda_p+1}{\lambda_p-1}+\frac{\lambda_p^3+1}{\lambda_p^3-1}(3+\nu)^2\right)\right)\right].
\]
where 
\[ D=\sqrt{2\left(1-\textrm{Re}\left({\lambda_p\cdot\partial_z\left(F^2(p)\right)}\right)\right)}.\] is of order $\sqrt{T'}$ and $\partial_z\left(F^2(p)\right)$ is the holomorphic derivative of $F^2$ with respect to $z=t+iv$. Notice that the fractions involving the multiplier $\lambda_p=\exp(i\theta_P)$ give respectively $i\cdot\cot(\theta_p/2)$ and $i\cdot\cot(3\theta_p/2)$. Of course $\theta_P$ is not independent of $\nu$, but recall that since we can control the multiplier, we can assume both cotangents to be bounded away from zero and positive. It is easy to check that for $\nu\in(-1,0)$ such polynomial is bounded away from zero, as each term in the sum is positive.
This is enough to establish the presence of an elliptic island around each periodic point found by the construction, proving Theorem B.
Notice that the expression for $\omega_3$ does not involve derivatives of $\varphi$ of order higher than 3 and derivatives of order 2 and higher of $T$ as such terms appear only in terms of lower order in the expansion in $T'$ (see below).
To conclude this section we now prove an estimate regarding the size of the elliptic islands we obtained.
\begin{proposition} Consider a 2-periodic orbit of type $(+)$, given by the points $(t_1,v_1)$ and $(t_2,v_2)$ and such that the multiplier is bounded away from resonances of order up to four; we define $T'=\left(T'(v_1)+T'(v_2)\right)/2$. Elliptic islands of type $(+)$ around such points have area of order $\left(T'\right)^{-3}$ for big enough $v$.
\label{measureElliptic}
\end{proposition}
\begin{proof}
We consider the map $F^2$ expressed in terms of the variables $u,\bar u$ defined above, close to a periodic point $p$; for simplicity we assume $p=0$. Recall that the variables $u$ and $\bar u$ are related to $z=t+iv$ and $\bar z$ by a linear symplectic transformation, i.e. $z=b_1u+b_2\bar u$ and $b_1\bar b_1-b_2\bar b_2=1$.
In such variables one can write the map as follows:
\[
u\mapsto A_1(u,\bar u)u + A_2(u,\bar u)\bar u \textrm{ where } A_1(0,0)=\lambda_p,\ A_2(0,0)=0.
\]
One can obtain all terms of the Taylor polynomial of $F^2$ in such variables by appropriately differentiating the functions $A_1$ and $A_2$ with respect to $u$ and $\bar u$. We claim that the term of order $n$ is of order at most $(T')^{3+(n-1)/2}$. By direct computation we find that 
\[
A_3, A_4, A_5 \sim (T')^{3/2}\quad A_6,A_7,A_8,A_9 \sim (T')^3.
\]
This, along with the estimate we claim, is sufficient to prove that the area of the elliptic island is of the required order. In fact, one can perform a rescaling $u\mapsto \Lambda u$, obtaining the following (symbolic) expression:
\[
u\mapsto \lambda_p u+\sum_{n=2}^\infty A^{(n)}\Lambda^{n-1}\{u,\bar u\}^n.
\]Therefore by choosing $\Lambda$ such that $A^{(n)}\Lambda^{n-1}\lesssim 1$ we obtain that the linearized part is dominant in a disk of radius of order $\Lambda$ around the origin. The explicit computations and the claim allows us to take $\Lambda\sim (T')^{-3/2}$; the result follows by recalling that the map $z\mapsto u$ is symplectic and therefore it preserves the area form.\\
We are now left with the proof of the claim, i.e. to prove that $A^{(n)}\lesssim(T')^{3+(n-1)/2}$. First we obtain by direct computation a relation between the coefficients $a_i$ of the Taylor expansion in terms of $z,\bar z$ and the coefficients $A_i$ of the expansion in terms of $u,\bar u$  
\[
z\mapsto a_1(z,\bar z) z +a_2(z,\bar z) \bar z\qquad z=b_1u+b_2\bar u.
\] 
\[
A_1=b_1\bar b_1 a_1 + \bar b_1 \bar b_2 a_2 - b_1 b_2 \bar a_2 - b_2 \bar b_2 \bar a_1\qquad A_2=\bar b_1 b_2 a_1 + \bar b_1 \bar b_1 a_2 - b_2 b_2 \bar a_2 - \bar b_1 b_2 \bar a_1.
\]
We can obtain all coefficients $A^{(n)}$ by applying the relative differential operator to the appropriate  $A_i$; the key fact to notice is that, $b_i$ being constant, the differential operator will operate only on the $a_i$. One can express $\partial_u$ and $\partial_{\bar u}$ in terms of $\partial_t$ and $\partial_{v}$ in the following way:
\[
\partial_u=\frac{b_1+\bar b_2}{2}\partial_t + i\frac{\bar b_2-b_1}{2}\partial_v \qquad \partial_{\bar u}=\frac{\bar b_1+b_2}{2}\partial_t - i\frac{ b_2-\bar b_1}{2}\partial_v.
\]
We are going to explicitly compute the coefficients of $\partial_t$ and $\partial_v$ to check that they are respectively of order $T'^{(1/2)}$ and $T'^{-(1/2)}$. Then we find a general expression for the order of arbitrary derivatives of $a_1$ and $a_2$.
Explicit calculations provide the following values for $a_i$:
\begin{eqnarray*}
a_1&=&\left(1+2\ddph(T'_1+T'_2)+2\ddph^2T'_1T'_2\right)+i\left(2\ddph+2\ddph^2T'_1-\frac{T'_1+T'_2}{2}-\ddph T'_1T'_2\right)\\
a_2&=&\left(-2\ddph T'_2-2\ddph^2T'_1T'_2\right)+i\left(2\ddph+2\ddph^2 T'_1+\frac{T'_1+T'_2}{2}+\ddph T'_1T'_2\right)
\end{eqnarray*}
Where recall that we defined $T'_i=T'(v_i)$; by defining as before
\[ D=\sqrt{2(1-\Re(\lambda_p\cdot a_1(0))}\sim T'^{(1/2)},\] and $T'=(1/2)(T'_1+T'_2)$ we compute $b_1$ and $b_2$:
\begin{eqnarray*}
b_1&=&\frac{1}{D}\left(2\ddph+2\ddph^2T'_1-T'-\ddph T'_1T'_2+\Im\lambda_p\right)\\
b_2&=&\frac{1}{D}\left(-2\ddph-2\ddph^2 T'_1-T'-\ddph T'_1T'_2-i\left(2\ddph T'_2+2\ddph^2 T'_1T'_2\right)\right).
\end{eqnarray*}  
From the above expression, recalling that $\ddph T'\sim 1$, we notice that each $b_i$ is of order $T'^{(1/2)}$. Now we compute the coefficients that appear in the change of variables $(\partial_t,\partial_v)\mapsto(\partial_u,\partial_{\bar u})$:
\begin{eqnarray*}
\frac{b_1+\bar b_2}{2}&=&\frac{1}{2D}\left(-2T'-2\ddph T'_1T'_2+\Im\lambda_p+i\left(2\ddph T'_2+2\ddph^2T'_1T'_2\right)\right)\\
\frac{\bar b_2 - b_1}{2}&=&\frac{1}{2D}\left(-4\ddph-4\ddph^2 T'_1-\Im\lambda_p+i\left(2\ddph T'_2+2\ddph^2T'_1T'_2\right)\right).
\end{eqnarray*}
From the expression we notice than the  coefficient of $\partial_t$ is of order $T'^{(1/2)}$ and the coefficient of $\partial_v$ is of order $T'^{-(1/2)}$. This reflects the fact that the symplectic transformation stretches along the $v$ direction and contracts along the $t$ direction in order to put the differential in normal form. This in turn implies that the shape of the invariant curves is elongated in the $t$ direction (as we can notice in figure \ref{imageIslands}).
We are left with computing the order of magnitude of derivatives of $a_1$ and $a_2$. It is convenient to notice that $d(F^2)$ is close to a square of a matrix, i.e. if we write $T'_1=T'(1-\delta)$ and $T'_2=T'(1+\delta)$ we obtain:
\[
\alpha_1\defeq 1+T'\ddph+i(\ddph-T'/2)\qquad \alpha_2\defeq -T'\ddph+i(\ddph+T'/2). 
\]
\begin{eqnarray*}
a_1&=&\left(\alpha_1\alpha_1+\alpha_2\bar\alpha_2\right)-2\ddph^2T'^2\delta^2-2i\ddph^2T'\delta+i\ddph T'^2\delta^2\\
a_2&=&\left(\alpha_2\alpha_1+\alpha_2\bar \alpha_1\right)-2\ddph T'\delta +2\ddph T'^2\delta^2-2i\ddph^2T'\delta-i\ddph T'^2\delta^2
\end{eqnarray*} 
As $\delta$ is of order 1 (and limiting to 0 as $T'\to\infty$), the error term is of order at most $T'$, whereas the main term is of order $T'^2$. Now we differentiate $\alpha_1$ and $\alpha_2$ with respect to $t$ and $v$:
\[
\partial_t\alpha_1=\dddph(T'+i)\qquad \partial_t\alpha_2=\dddph(-T'+i)
\]
\[
\partial_v\alpha_1=T'\partial_t\alpha_1+T''(\ddph-i/2) \qquad \partial_v\alpha_2=T'\partial_t\alpha_2+T''(-\ddph+i/2).
\]
In order to obtain an upper bound on such derivatives, we will consider $\phi^{(n)}\sim 1$ regardless of the fact that even derivatives will be of order $T'^{-1}$. To this extent, we observe that all terms containing second (and higher) derivatives of $T$ will appear in terms of lower order than the dominant $T'$ for $\partial_t$ and $T'^2$ for $\partial_v$. By direct inspection, the same statement is true for terms containing $\delta$ in the expression for $a_i$ (in fact $\delta\sim T''/T'$). Therefore if we restrict to the maximum order:
\[
\left.\partial_t^k\partial_v^l\alpha_1\right|_{max}=T'^l\phi^{(k+l+2)}(T'+i)\qquad\left.\partial_t^k\partial_v^l\alpha_2\right|_{max}=T'^l\phi^{(k+l+2)}(-T'+i).
\] 
Now recall that we were to compute derivatives with respect to the $(u,\bar u)$ variables and as such we should recall that the coefficients in front of $\partial_t$ and $\partial_v$ are of order respectively $T'^{(1/2)}$ and $T'^{(-1/2)}$. This means that we obtain:
\[
\partial_u^k\partial_{\bar u}^l\alpha_1\lesssim T'^{(k+l)/2+1}\qquad\partial_u^k\partial_{\bar u}^l\alpha_2\lesssim T'^{(k+l)/2+1},
\] 
which in turn implies:
\[
\partial_u^k\partial_{\bar u}^la_1\lesssim T'^{(k+l)/2+2}\qquad\partial_u^k\partial_{\bar u}^la_2\lesssim T'^{(k+l)/2+2}. 
\]
Therefore we obtain the required estimate for $A^{(n)}$, i.e.:
\[A^{(n)}\lesssim T' \partial^{(n-1)} a_i\lesssim T'^{3+(n-1)/2},\]
which concludes the proof.\end{proof}
\label{nondegeneracy}
\begin{acknowledgements}
I would like to thank my advisor Dmitrij Dolgopyat for introducing me to the problem and for most useful discussions and comments
\end{acknowledgements}

\begin{figure}[!h]
\includegraphics[width=10.0cm]{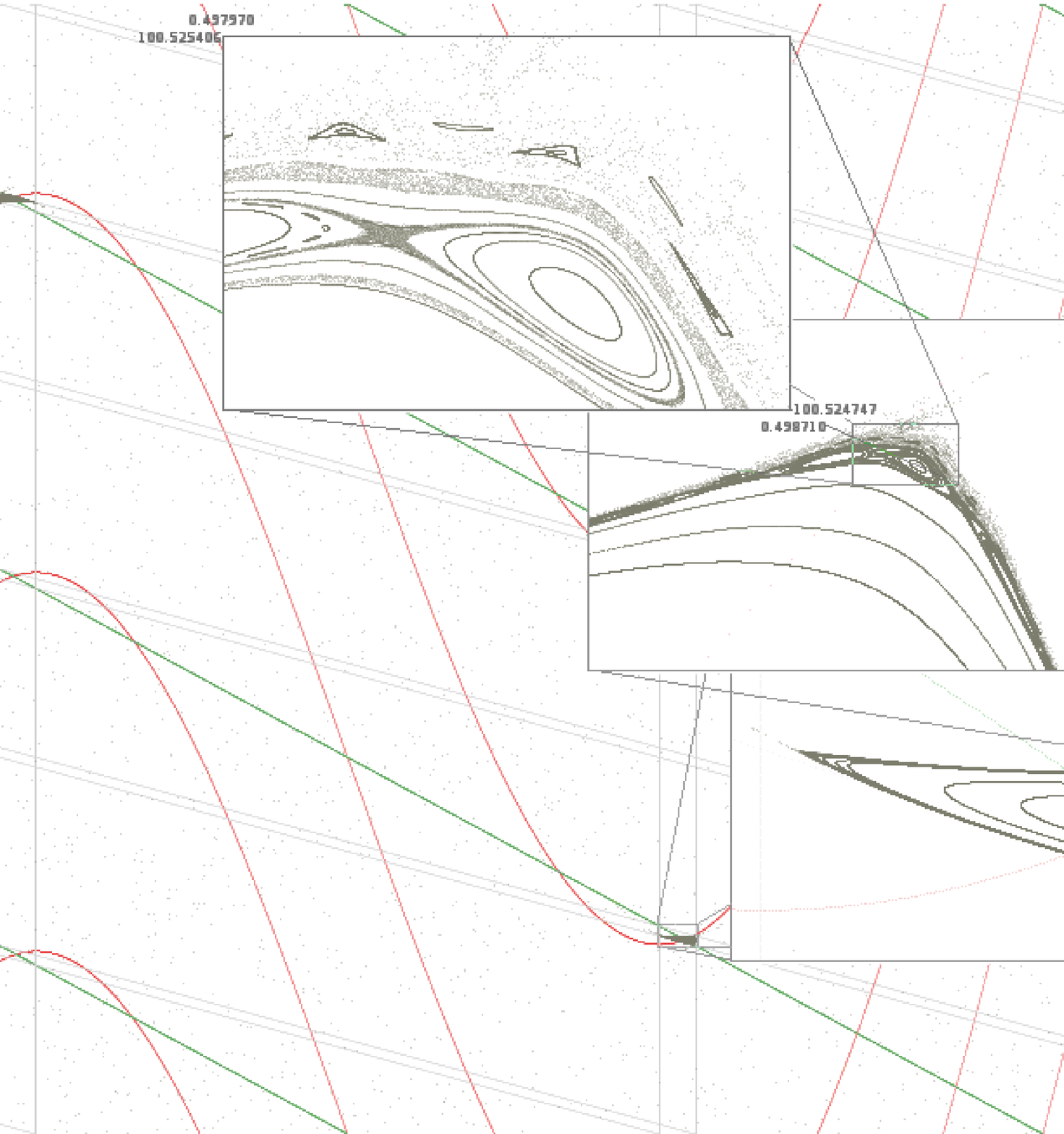}
\caption{Elliptic island of period 2 and type $(+)$; each smaller picture is an enlarged portion of the previous one. In the big picture we can notice the 2-periodic islands (bottom center and upper left) at two suitable intersection points of $\ell$ and $\ell^1$.}\label{imageIslands}
\end{figure}
\end{document}